\documentclass[a4paper,10pt]{article}
\usepackage{amssymb}
\usepackage{amsmath}
\usepackage{amsthm}
\usepackage{graphicx}

\usepackage{epstopdf}
 \usepackage[colorlinks=true]{hyperref}
\usepackage{pgf,tikz}
\hypersetup{urlcolor=blue, citecolor=red}
  \usepackage{paralist}

\newtheorem{theorem}{Theorem}[section]
\newtheorem{lemma}[theorem]{Lemma}

\newtheorem{rmk}{Remark}
\newtheorem{exa}{Example}

\numberwithin{equation}{section}

\newcommand{\ep}{\varepsilon}

\newcommand{\va}{\varphi}
\newcommand{\ppp}{\partial}

\newcommand{\weight}{e^{2s\va}}

\newcommand{\R}{\mathbb{R}}

\newcommand{\www}{\widetilde}
\newcommand{\OOO}{\Omega}
\newcommand{\ooo}{\overline}
\newcommand{\ddd}{\mbox{div}\thinspace}

\newcommand{\be}{\begin{equation}}
\newcommand{\bel}{\begin{equation} \label}
\newcommand{\ee}{\end{equation}}
\newcommand{\pd}{\partial}
\def\beq{\begin{equation}}
\def\eeq{\end{equation}}
\newcommand{\bea}{\begin{eqnarray}}
\newcommand{\eea}{\end{eqnarray}}
\newcommand{\beas}{\begin{eqnarray*}}
\newcommand{\eeas}{\end{eqnarray*}}

\renewcommand{\div}{\mathrm{div}\,}  

\allowdisplaybreaks

\graphicspath{
  {FIGURES/PNG/}
}

\title{Lipschitz stability for an inverse hyperbolic problem of
determining two coefficients by a finite number of observations
}

\author{L. Beilina  \thanks{
Department of Mathematical Sciences, Chalmers University of Technology and
 University  of Gothenburg, SE-42196 Gothenburg, Sweden, e-mail: \texttt{\
larisa@chalmers.se}}
\and
M. Cristofol \thanks{Institut de Math\'{e}matiques de Marseille, CNRS,
UMR 7373, \'Ecole Centrale, Aix-Marseille Universit\'e, 13453 Marseille,
France,   e-mail: \texttt{\ michel.cristofol@univ-amu.fr} }
\and
S. Li \thanks{Key Laboratory of Wu Wen-Tsun Mathematics, Chinese Academy of
Sciences, School of Mathematical Sciences, University of Science and
Technology of China, 96 Jinzhai Road, Hefei, Anhui Province, 230026,
China, e-mail: \texttt{\ shuminli@ustc.edu.cn}}
\and
M. Yamamoto \thanks{Department of Mathematical Sciences, The University of
Tokyo, Komaba, Meguro, Tokyo 153, Japan, e-mail: {\tt myama@ms.u-tokyo.ac.jp
}}}
\date{}
\begin{document}
\maketitle

\baselineskip 18pt

\begin{abstract}
We consider an inverse problem of reconstructing two spatially
varying coefficients in an acoustic equation of hyperbolic type
  using interior data of solutions with suitable choices of initial
condition.
Using a Carleman estimate, we prove Lipschitz stability estimates
which ensures
unique reconstruction of both coefficients.
Our theoretical results are justified by numerical studies on the
reconstruction of two unknown coefficients using  noisy backscattered data.

\end{abstract}

\section{Statement of the problem}\label{S1}
\subsection{Introduction}
 The main purpose of this paper is
  to study the inverse problem of determining simultaneously the
  function $\rho(x)$ and the
  conductivity $p(x)$ in the following:
   \bel{11} \rho(x) \pd_t^2
  u - \div (p(x) \nabla u )= 0 \ee
  from a finite number of
  boundary observations on the domain $\Omega$ which is a bounded open
  subset of $\R^n,\; n\geq1$.\\

  The reconstruction of two coefficients
  of the principal part of an operator with a finite number of
  observations is very challenging since we mix at least two
  difficulties, see \cite{BJY08} for the case of a principal matrix
  term in the divergence form, arising from anisotropic media) or
  \cite{ImIsY} for Lame system or \cite{BCN, BCS12, L05,LY05,LY07} for
  Maxwell system.\\

 Furthermore, in this work we establish  a Lipschitz stability inequality.
 First, this stability inequality implies the uniqueness
 of the reconstruction of coefficients $\rho(x)$ and $p(x)$. Second, we can use
 it to perform numerical reconstruction with noisy observations to be more
 close to real-life applications.  \\

Bukhgeim and Klibanov \cite{BK} created the methodology by Carleman estimate
for proving the uniqueness in coefficient inverse problems and after \cite{BK},
there has been many works published  on this topic. We refer to some of them.
\cite{Bel1, Bel2, BJY08, BelY1, BeYa08}, \cite{IY2} - \cite{IY4},
\cite{Kli1} - \cite{Kli3}, \cite{KY, Y99}.
In all these works except the recent works \cite{BCL, BCN}, only theoretical
studies are presented. From other side, the existence of a stability
theorems allow us to improve the results of the numerical
reconstruction by choosing different regularization strategies in
the minimization procedure.

In particular we refer to Imanuvilov and Yamamoto \cite{IY3} which established
the Lipschitz stability for the coefficient inverse problem for a hyperbolic
equation. Our argument in this paper is a simplification of \cite{IY3} and
Klibanov and Yamamoto \cite{KY}.

  To the authors' knowledge, there exist few works which study
  numerical reconstruction based on the theoretical stability analysis
  for the inverse problem with finite and restricted measurements.
  Furthermore, the case of the reconstruction of the conductivity
  coefficient in the divergence form for the hyperbolic operator
  induces some numerical difficulties, see \cite{BAbsorb, BH, BNAbsorb, Chow}
for details.

In numerical simulations of this paper we use similar optimization
approach which was applied recently in works
\cite{BAbsorb, BCL, BCN, BJ, BNAbsorb}.
More precisely, we minimize the Tikhonov functional
in order to reconstruct unknown spatially distributed wave speed and
conductivity functions of the acoustic wave equation from transmitted
or backscattered boundary measurements.  For minimization of the
Tikhonov functional we construct the associated Lagrangian and
minimize it on the adaptive locally refined meshes using the domain
decomposition finite element/finite difference method similar to one
of \cite{BAbsorb}. Details of this method can be found in forthcoming
publication.
   The adaptive optimization method  is implemented
efficiently in
   the software package WavES \cite{waves} in  C++/PETSc
   \cite{petsc}.\\

Our numerical simulations show that we can accurately
   reconstruct location of both space-dependent wave speed and
   conductivity functions already on a coarse non-refined mesh.  The
   contrast of the conductivity function is also reconstructed
   correctly.  However, the contrast of the wave speed function should
   be improved.  In order to obtain better contrast, similarly with
   \cite{BMaxwell, BH, BJ}, we applied an adaptive finite element
   method, and refined the finite element mesh locally only in places,
   where the a posteriori error of the reconstructed coefficients was
   large.
   Our final results attained
   on a locally refined meshes show that an adaptive finite element
   method significantly improves reconstruction obtained on a coarse
   mesh. \\

   The outline of this paper is as follows.  In Section \ref{S2}, we show
 a key Carleman estimate, in Section \ref{S3} we complete the proofs of
 Theorems \ref{T1} and \ref{T2}. Finally, in section \ref{S4} we present
 numerical simulations taking into account the theoretical observations
 required in Theorem \ref{T1} as an important guidance. Section
 \ref{S5} concludes the main results of this paper.

 \subsection{Settings and main results}

Let $\Omega \subset \R^n$ be a bounded domain with smooth boundary
$\ppp\Omega$.  We consider an acoustic equation
\bel{1.1}
\rho(x)\ppp_t^2 u(x,t) - \ddd (p(x)\nabla u(x,t)) = 0, \quad x \in \OOO,
\thinspace 0 < t < T.           
\ee
To \eqref{1.1} we attach the initial and boundary conditions:
\bel{1.2}
u(x,0) = a(x), \quad \ppp_tu(x,0) =0, \quad x\in \OOO
\ee
and
\bel{1.3}
u(x,t) = h(x,t), \quad (x,t) \in \ppp\OOO \times (0,T).
\ee
We will write $u(p,\rho,a, h)$ a weak solution of the problem
\eqref{1.1}-\eqref{1.3}.
Functions $p, \rho$ are assumed to be positive on $\ooo{\OOO}$ and
are unknown  in $\Omega$. They should be determined by extra data of solutions
$u$ in $\Omega$.

Throughout this paper, we set $\ppp_j = \frac{\ppp}{\ppp x_j}$,
$\ppp_i\ppp_j = \frac{\ppp^2}{\ppp x_i\ppp x_j}$,
$\ppp_t^2 = \frac{\ppp^2}{\ppp t^2}$, $1 \le i,j \le n$.

Let $\omega \subset \OOO$ be a suitable subdomain of $\OOO$ and
$T > 0$ be given. In this paper, we consider an inverse problem of
determining coefficients $p = p(x)$ and $\rho = \rho(x)$ of
the principal term, from the interior observations:
$$
u(x,t), \quad x \in \omega, \thinspace 0 < t < T.
$$

In order to formulate our results, we need to introduce some notations.
For sufficiently smooth positive coefficients $p$ and $\rho$ and initial
and boundary data, we can prove the existence of a unique
weak solution to \eqref{1.1}-\eqref{1.3} (e.g., Lions and Magenes
\cite{LG72}), which we denote by $u = u(p,\rho,a,h)$.

Henceforth $(\cdot,\cdot)$ denotes the scalar product in $\R^n$, and
$\nu = \nu(x)$ be the unit outward normal vector to $\ppp\OOO$ at $x$.
Let the subdomain $\omega \subset \OOO$ satisfy
\bel{1.4}
\ppp\omega \supset \{ x\in \ppp\OOO; \thinspace ((x-x_0)\cdot \nu(x))) > 0\}
\ee
with some $x_0 \not\in \ooo{\OOO}$.
We note that $\omega \subset \OOO$ cannot be an arbitrary subdomain.
For example, in the case of a ball $\OOO$, the condition \eqref{1.4}  requires
that $\omega$ should be a neighborhood of a sub-boundary which is larger
than the half of $\ppp\OOO$.
The condition \eqref{1.4} is also a sufficient condition for an observability
inequality by observations in $\omega \times (0,T)$ (e.g.,
Ch VII, section 2.3 in Lions \cite{L88}).

We set
\bel{1.5}
\Lambda = \left( \sup_{x\in \OOO} \vert x-x_0\vert^2
- \inf_{x\in \OOO} \vert x-x_0\vert^2\right)^{\frac{1}{2}}.
\ee
We define admissible sets of
unknown coefficients.  For arbitrarily fixed functions $\eta_0
\in C^2(\ooo{\OOO})$, $\eta_1 \in (C^2(\ooo{\OOO}))^n$ and
constants $M_1 > 0, 0<\theta_0\le 1,
\theta_1>0$, we set
\bel{1.6}
\mathcal{U}^1 = \mathcal{U}^1_{M_1, \theta_1,\eta_0,\eta_1}
= \biggl\{ p \in C^2(\ooo{\OOO}); \thinspace
p = \eta_0, \thinspace \nabla p = \eta_1 \quad \mbox{on $\ppp\OOO$},
\ee
\begin{align*}
& \Vert p\Vert_{C^2(\ooo{\OOO})} \le M_1,\quad
p \ge \theta_1 \quad \mbox{on $\ooo{\OOO}$} \biggr\}, \\
&
\mathcal{U}^2 = \mathcal{U}^2_{M_1, \theta_1}
= \biggl\{ \rho \in C^2(\ooo{\OOO}); \thinspace
\Vert \rho\Vert_{C^2(\ooo{\OOO})} \le M_1,\quad
\rho \ge \theta_1 \quad \mbox{on $\ooo{\OOO}$} \biggr\}, \\
& \mathcal{U} = \mathcal{U}_{M_1,\theta_0,\theta_1,\eta_0,\eta_1,x_0}
= \biggl\{(p,\rho) \in \mathcal{U}^1\times \mathcal{U}^2;
\quad \frac{(\nabla(p\rho^{-1})\cdot (x-x_0))}
{2p\rho^{-1}(x)} < 1 - \theta_0, \quad x \in \ooo{\OOO\setminus \omega}
\biggr\}.
\end{align*}

We note that there exists a constant $M_0>0$ such that
$\left\Vert \nabla\left( \frac{p}{\rho}\right)
\right\Vert_{C(\ooo{\OOO})} \le M_0$ for each $(p,\rho) \in
\mathcal{U}^1 \times \mathcal{U}^2$.
 \quad

Then we choose a constant $\beta > 0$ such that
\bel{1.7}
\beta + \frac{M_0\Lambda}{\sqrt{\theta_1}}\sqrt{\beta} < \theta_0\theta_1,
\quad \theta_1\inf_{x\in \OOO}\vert x-x_0\vert^2 - \beta\Lambda^2 > 0.
\ee
Here we note that such $\beta>0$ exists by $x_0 \not\in \ooo{\OOO}$, and
in fact $\beta > 0$ should be sufficiently small.

We are ready to state our first main result.
\\
\vspace{0.2cm}
\begin{theorem}\label{T1}
Let $q \in \mathcal{U}^1$ be arbitrarily fixed and
let $a_1, a_2 \in C^3(\ooo{\OOO})$ satisfy
\bel{1.8}
\left\{ \begin{array}{rl}
& \vert \ddd (q\nabla a_{\ell}) \vert > 0, \quad \mbox{$\ell=1$ or
$\ell=2$,} \\
& ((\ddd (q\nabla a_2)\nabla a_1 - \ddd (q\nabla a_1)\nabla a_2)
\cdot (x- x_0)) > 0 \quad \mbox{on $\ooo{\OOO}$}.
\end{array}\right.
\ee
We further assume that
$$
u(q,\sigma,a_{\ell}, h_{\ell})
\in W^{4,\infty}(\OOO\times (0,T)), \quad \ell = 1,2
$$
and
\bel{1.9}
T > \frac{\Lambda}{\sqrt{\beta}}.                      
\ee
Then there exists a constant $C>0$ depending on $\OOO, T, \mathcal{U},
q, \sigma$ and a constant $M_2>0$ such that
\bel{1.10}
\Vert p-q\Vert_{H^1(\OOO)} + \Vert \rho - \sigma\Vert_{L^2(\OOO)}
\le C\sum_{\ell=1}^2 \Vert u(p,\rho,a_{\ell}, h_{\ell}) -
u(q,\sigma,a_{\ell}, h_{\ell})\Vert_{H^3(0,T;L^2(\omega))}
\ee
for each $(p,\rho) \in \mathcal{U}$ satisfying
\bel{1.11}
\Vert u(p,\rho,a_{\ell}, h_{\ell}) \Vert_{W^{4,\infty}(\OOO\times (0,T))}
\le M_2.                              
\ee
\end{theorem}

The conclusion \eqref{1.10} is a Lipschitz stability estimate with twice
changed
initial displacement satisfying \eqref{1.8}.  In Imanuvilov and Yamamoto
\cite{IY4}, by assuming that $\rho = \sigma \equiv 1$, a H\"older stability
estimate is proved for $p-q$, provided that $p$ and $q$ vary within a similar
admissible set.  However, in the case of two
unknown coefficients $p, \rho$, the condition \eqref{1.8}
requires us to fix $q \in \mathcal{U}^1$ and the theorem gives stability
only around given $q$, in general.

\begin{rmk}
In this remark, we will show that with special choice of $a_1, a_2$,
the condition \eqref{1.8} can be satisfied uniformly for
$q \in \mathcal{U}^1$, which guarantees that the set of $a_1, a_2$
satisfying (1.9), is not empty.  We fix $a_1, b_2 \in C^2(\ooo{\OOO})$
satisfying
\bel{1.12}
(\nabla a_1(x) \cdot (x-x_0)) > 0, \quad
\vert \nabla b_2(x)\vert > 0, \quad x \in \ooo{\OOO}.  
\ee
We choose $\gamma > 0$ sufficiently large and we set
$$
a_2(x) = e^{\gamma b_2(x)}.
$$
Then $\ppp_ka_2 = \gamma(\ppp_kb_2)e^{\gamma b_2(x)}$ and
$$
\Delta a_2 = (\gamma^2\vert \nabla b_2\vert^2 + \gamma\Delta b_2)
e^{\gamma b_2},
$$
and so
$$
\ddd (q\nabla a_2) = q\Delta a_2 + \nabla q\cdot \nabla a_2
= (q\gamma^2 \vert \nabla b_2\vert^2 + O(\gamma))e^{\gamma b_2}
$$
and
\begin{align*}
& (\ddd (q\nabla a_2)\nabla a_1 - \ddd (q\nabla a_1)\nabla a_2)
\cdot (x- x_0)) \\
=& e^{\gamma b_2}\{ (q\gamma^2 \vert \nabla b_2\vert^2 + O(\gamma))
\nabla a_1 - \ddd (q\nabla a_1) \gamma\nabla b_2\} \cdot (x-x_0)\\
=& e^{\gamma b_2}(
q\gamma^2 \vert \nabla b_2\vert^2(\nabla a_1 \cdot (x-x_0))
+ O(\gamma))\\
\ge & e^{\gamma \min_{x\in\overline{\Omega}} b_2(x)}
(\gamma^2\theta_1 \min_{x\in \overline{\Omega}}
\{\vert \nabla b_2(x)\vert^2(\nabla a_1(x)\cdot (x-x_0))\}
+ O(\gamma))
\end{align*}
for each $q\in \mathcal{U}^1$.
Therefore, for large $\gamma > 0$, by \eqref{1.12} we see that \eqref{1.8} is
fulfilled.
Moreover this choice of $a_1, a_2$ is independent of choices of
$q \in \mathcal{U}^1$, and there exists a constant $C>0$, which is dependent
on $\Omega, T, \mathcal{U}, M_2$ but independent of choices
$(p,\rho), (q,\sigma)$, such that (1.11) holds for each
$(p,\rho), (q,\sigma) \in \mathcal{U}$.
\end{rmk}

\vspace{0.3cm}

Without special choice such as (1.13), we consider the stability estimate
by not fixing $q$.
If we can suitably choose initial values $(n+1)$-times, then
we can establish the Lipschitz stability for arbitrary
$(p, \rho), (q,\sigma) \in \mathcal{U}$.
\\
\begin{theorem}\label{T2}
Let $A :=
\left( \begin{array}{cc}
a_1\\
\vdots \\
a_{n+1}\\
\end{array}\right)
\in (C^2(\ooo{\OOO}))^{n+1}$ satisfy
\bel{1.13}
\mbox{det} \thinspace (\ppp_1A(x), ..., \ppp_nA(x), \Delta A(x))
\ne 0, \quad x \in \ooo{\OOO}.              
\ee
We assume \eqref{1.9}.  Then there exists a constant $C>0$ depending on
$\OOO, T, \mathcal{U}, a_{\ell}, h_{\ell}$, $\ell = 1, 2,..., n+1$
and a constant $M_2>0$ such that
\bel{1.14}
\Vert p-q\Vert_{H^1(\OOO)} + \Vert \rho - \sigma\Vert_{L^2(\OOO)}
\le C\sum_{\ell=1}^{n+1} \Vert u(p,\rho,a_{\ell}, h_{\ell})
- u(q,\sigma,a_{\ell}, h_{\ell})\Vert_{H^2(0,T;L^2(\omega))}
\ee
for each $(p,\rho), (q,\sigma) \in \mathcal{U}$ satisfying
$$
\Vert u(p,\rho,a_{\ell}, h_{\ell}) \Vert_{W^{4,\infty}(\OOO\times (0,T))},
\Vert u(q,\sigma,a_{\ell}, h_{\ell}) \Vert_{W^{4,\infty}(\OOO\times (0,T))}
\le M_2, \quad \ell=1, 2, ..., n+1.
$$
\end{theorem}
\begin{exa}
This example illustrates how to choose initial values satisfying \eqref{1.13}.
Although in Theorem \ref{T2} , we have to take more observations, the condition for
the initial values is more generous compared with Theorem \ref{T1}.
For example, we can choose the following initial displacement
$a_1, ..., a_{n+1}$: let $D= (d_{ij})_{1\le i,j \le n}$ be a matrix such that
$d_{ij} \in \R$ and $D^{-1}$ exists.  Then we give linear functions
$a_1, ..., a_n$ by
$$
a_{\ell}(x) = \sum_{k=1}^n d_{\ell k}x_k, \quad \ell=1, 2, ..., n
$$
and we choose $a_{n+1}(x)$ satisfying $\Delta a_{n+1}(x) \ne 0$ for
$x \in \ooo{\OOO}$.  Then we can easily verify that this choice
$a_1, ..., a_{n+1}$ satisfies \eqref{1.13}.
\end{exa}
\vspace{0.2cm}

We note that Theorems \ref{T1} and \ref{T2}  yield the uniqueness for our
inverse problem in the respective case.

\section{The Carleman estimate for a hyperbolic equation}\label{S2}

We show a Carleman estimate for a second-order hyperbolic equation.
We recall that $\mathcal{U}$ is defined by \eqref{1.6}.

Let us set
$$
Q = \OOO \times (-T,T).
$$
For $x_0 \not\in \ooo{\OOO}$ and $\beta > 0$ satisfying \eqref{1.7},
we define the functions $\psi = \psi(x,t)$ and $\va = \va(x,t)$ by
\bel{2.1}
\psi(x,t) = \vert x-x_0\vert^2 - \beta t^2       
\ee
and
\bel{2.2}
\va(x,t) = e^{\lambda\psi(x,t)}              
\ee
with parameter $\lambda > 0$.  We add a constant $C_0 > 0$ if necessary
so that we can assume that $\psi(x,t) \ge 0$ for $(x,t) \in Q$, so that
$$
\va(x,t) \ge 1, \quad (x,t) \in \ooo{Q}.
$$

Henceforth $C>0$ denotes generic constants which are independent of
parameter $s>0$ in the Carleman estimates and choices of $(p,\rho),
(q, \sigma) \in \mathcal{U}$.

We show a Carleman estimate which is derived from Theorem \ref{T2} in
Imanuvilov \cite{Ima2}. See Imanuvilov and Yamamoto \cite{IY4} for
a concrete sufficient condition on the coefficients yielding
a Carleman estimate.
\\
\vspace{0.2cm}
\begin{lemma}\label{L1}
We assume $(\mu,1) \in \mathcal{U}$, and that \eqref{1.4} holds for some
$x_0 \not\in \ooo{\OOO}$.  Let $y \in H^1(Q)$ satisfy
\bel{2.3}
\ppp_t^2y(x,t) - \mu\Delta y = F \quad \mbox{in $Q$}  
\ee
and
\bel{2.4}
y(x,t) = 0, \quad (x,t)\in \ppp\OOO\times (-T,T), \quad
\ppp_t^ky(x,\pm T) = 0, \quad x\in \OOO, \thinspace k=0,1.  
\ee
Let
\bel{2.5}
T > \frac{\Lambda}{\sqrt{\beta}}.              
\ee
We fix $\lambda>0$ sufficiently large.  Then there exist constants
$s_0 > 0$ and $C>0$ such that
\bel{2.6}
\int_Q (s\vert \nabla_{x,t}y\vert^2 + s^3\vert y\vert^2)\weight dxdt
\le C\int_Q \vert F\vert^2\weight dxdt
+ C\int^T_{-T}\int_{\omega} (s\vert \ppp_ty\vert^2
+ s^3\vert y\vert^2) \weight dxdt                 
\ee
for all $s > s_0$.
\end{lemma}

In the Lemma \ref{L1}, we notice that the constants $C>0$ and $s_0>0$ are
determined by $\mathcal{U}, \OOO, T, x_0, \omega$ and independent of
$s$ and choices of the coefficients $(\mu,1), (p,\rho), (q, \sigma)
\in \mathcal{U}$.

Setting $\Gamma = \{ x\in \ppp\OOO;\thinspace
(x-x_0) \cdot \nu(x) \ge 0\}$, one can prove a Carleman estimate whose
second term on the right-hand side of (2.6) is replaced by
$$
\int^T_{-T} \int_{\Gamma} s\vert \nabla y\cdot \nu\vert^2
\weight dSdt,
$$
and as for a direct proof, see Bellassoued and Yamamoto \cite{BY},
Cheng, Isakov, Yamamoto and Zhou \cite{ChIsYZh}.
In Isakov \cite{Is1}, a similar Carleman estimate is established
for supp $y \subset Q$, which cannot be applied to the case
where we have no Neumann data outside of $\Gamma$.

For the Carleman estimate, we have to assume that
$\ppp_t^ky(\cdot,\pm T) = 0$ in $\OOO$ for $k=0,1$, but
$u(p,\rho,a,h)$, $u(q,\sigma,a,h)$ do not satisfy this condition.
Thus we need a cut-off function which is defined as follows.

By \eqref{1.9} and the definitions \eqref{2.1} and \eqref{2.2} of $\psi, \va$,
we can choose $d_0 \in \R$ such that
\bel{2.7}
\va(x,0) > d_0, \quad \va(x,\pm T) < d_0, \qquad x \in \ooo{\OOO}.
\ee
Hence, for small $\ep_0 > 0$, we find a sufficiently small $\ep_1>0$ such
that
\bel{2.8}
\va(x,t) \ge d_0 + \ep_0, \quad (x,t) \in \ooo{\OOO \times [-\ep_1,\ep_1]}
\ee
and
\bel{2.9}
\va(x,t) \le d_0 - \ep_0, \quad (x,t) \in \ooo{\OOO} \times
([-T, -T+2\ep_1] \cup [T-2\ep_1, T]).           
\ee
We define a cut-off function satisfying $0 \le \chi \le 1$,
$\chi \in C^{\infty}(\R)$ and
\bel{2.10}
\chi(t) =
\left\{ \begin{array}{rl}
0, \qquad & -T \le t \le -T+\ep_1, \quad T-\ep_1 \le t \le T, \\
1, \qquad &-T+2\ep_1 \le t\le T-2\ep_1.
\end{array}\right.                      
\ee
Henceforth we write $\chi'(t) = \frac{d\chi}{dt}(t)$,
$\chi''(t) = \frac{d^2\chi}{dt^2}(t)$.

In view of the cut-off function, we can prove

\begin{lemma}\label{L2}
Let $(p,\rho) \in \mathcal{U}$ and let \eqref{2.5} hold, and we fix
$\lambda>0$ sufficiently large.  Then there exist constants
$s_0 > 0$ and $C>0$ such that
$$
\int_Q (s\vert \nabla_{x,t}u\vert^2 + s^3\vert u\vert^2)\weight dxdt
\le C\int_Q \vert \rho\ppp_t^2u - \ddd (p\nabla u)\vert^2\weight dxdt
 $$
\bel{2.11}
+ Cs^3e^{2s(d_0-\ep_0)}\Vert u\Vert^2_{H^1(Q)}
+ C\int^T_{-T}\int_{\omega} (s\vert \ppp_tu\vert^2
+ s^3\vert u\vert^2) \weight dxdt
\ee
for all $s > s_0$ and $u \in H^1(Q)$ satisfying
$\rho\ppp_t^2u - \ddd (p\nabla u) \in L^2(Q)$ and
$u\vert_{\ppp\OOO} = 0$.
\end{lemma}

\begin{proof}
  We notice
$$
u = \chi u + (1-\chi)u.
$$
Then
\begin{align*}
&\int_Q (s\vert \nabla_{x,t}u\vert^2 + s^3\vert u\vert^2)\weight dxdt\\
\le& 2\int_Q (s\vert \nabla_{x,t}(\chi u) \vert^2
+  s^3\vert \chi u\vert^2)\weight dxdt
+ 2\int_Q (s\vert \nabla_{x,t}((1-\chi)u)\vert^2
+ s^3\vert (1-\chi) u\vert^2)\weight dxdt.
\end{align*}
Since the second term on the right-hand side does not vanish only if
$T-2\ep_1 \le \vert t\vert \le T$, that is,
only if $\va(x,t) \le d_0 - \ep_0$ by \eqref{2.9}, we obtain
$$
\int_Q (s\vert \nabla_{x,t}u\vert^2 + s^3\vert u\vert^2)\weight dxdt
$$
\bel{2.12}
\le 2\int_Q (s\vert \nabla_{x,t}(\chi u) \vert^2
+  s^3\vert \chi u\vert^2)\weight dxdt
+ Cs^3e^{2s(d_0-\ep_0)}\Vert u\Vert^2_{H^1(Q)}.
\ee
On the other hand, we have
$$\left\{ \begin{array}{rl}
& \ppp_t^2(\chi u)(x,t) = \frac{p}{\rho}\Delta (\chi u)+\frac{\chi}{\rho}(\rho\ppp_t^2u - \ddd (p\nabla u))
+ \frac{\nabla p}{\rho}\cdot \nabla (\chi u)
+ 2\chi' \ppp_tu + \chi'' u \quad \mbox{in $Q$}, \\
& \chi u\vert_{\ppp\OOO} = 0, \\
& \ppp_t^k(\chi u)(\cdot, \pm T) = 0 \quad \mbox{in $\OOO$, $k=0,1$.}\\
\end{array}\right.
$$
Therefore, applying Lemma \ref{L1} to $\left(\ppp_t^2 - \frac{p}{\rho}\Delta\right)
(\chi u)$ by regarding $\frac{\chi}{\rho}(\rho\ppp_t^2u - \ddd (p\nabla u))$
 $+\frac{\nabla p}{\rho}\cdot \nabla (\chi u) + 2\chi' \ppp_tu + \chi'' u$
as non-homogeneous term, and choosing $s>0$ sufficiently large, we obtain
\begin{align*}
&\int_Q (s\vert \nabla_{x,t}(\chi u) \vert^2
+  s^3\vert \chi u\vert^2)\weight dxdt\\
\le & C\int_Q \vert \frac{\chi}{\rho}(\rho\ppp_t^2u - \ddd (p\nabla u))\vert^2
\weight dxdt\\
+ & C\int_Q \vert 2\chi'\ppp_tu + \chi'' u\vert^2 \weight dxdt
+ C\int^T_{-T}\int_{\omega}
(s\vert \ppp_t(\chi u)\vert^2 + s^3\vert \chi u\vert^2)\weight dxdt\\
\le& C\int_Q \vert \rho\ppp_t^2u - \ddd (p\nabla u)\vert^2\weight dxdt \\
+ & Ce^{2s(d_0-\ep_0)}\Vert u\Vert^2_{H^1(Q)}
+ C\int^T_{-T}\int_{\omega}
(s\vert \ppp_t u\vert^2 + s^3\vert u\vert^2)\weight dxdt.
\end{align*}
At the last inequality, we used the same argument as the second term on the
right-hand side of \eqref{2.12}. Substituting this in the first term on the
right-hand side of \eqref{2.12}, we complete the proof of Lemma \ref{L2}.

\end{proof}

We conclude this section with a Carleman estimate for a first-order
partial differential equation.
\\
\begin{lemma}\label{L3}
Let $A \in (C^1(\ooo{\OOO}))^n$ and $B \in C^1(\ooo{\OOO})$, and let
$$
Qf := A(x)\cdot \nabla f(x) + B(x)f, \quad f\in H^1(\OOO).
$$
We assume
\bel{2.13}
(A(x) \cdot (x-x_0)) \ne 0, \quad x \in \ooo{\OOO}.  
\ee
Then there exist constants $s_0>0$ and $C>0$ such that
\bel{2.14}
\int_{\OOO} s^2\vert f\vert^2 e^{2s\va(x,0)} dx
\le C\int_{\OOO} \vert Qf\vert^2 e^{2s\va(x,0)} dx  
\ee
for $s>s_0$ and $f \in H^1_0(\Omega)$ and
\bel{2.15}
\int_{\OOO} s^2(\vert f\vert^2 + \vert \nabla f\vert^2)e^{2s\va(x,0)} dx
\le C\int_{\OOO} (\vert Qf\vert^2 + \vert \nabla (Qf)\vert^2)
e^{2s\va(x,0)} dx                         
\ee
for $s>s_0$ and $f \in H^2_0(\Omega)$.
\end{lemma}

The proof can be done directly by integration by parts, and we refer for
example to Lemma 2.4 in Bellassoued, Imanuvilov and Yamamoto \cite{BIY}.
\section{Proofs of Theorems \ref{T1} and \ref{T2}}\label{S3}
\subsection{Proof of Theorem \ref{T1}}

We divide the proof into three steps.
The argument in Second Step is a simplification of the corresponding
part in \cite{IY3}, while the energy estimate \eqref{3.16} in
Third Step modifies the argument towards the Lipschitz stability in
\cite{KY}.

{\bf First Step: Even extension in $t$.}

We set
$$
y(a)(x,t) = u(p,\rho,a,h)(x,t) - u(q,\sigma,a,h)(x,t),
\quad R(x,t) = u(q,\sigma,a,h)(x,t),
$$ and we write $y$ in place of $y(a)$. We define
\bel{3.1}
f(x) = p(x) - q(x), \quad g(x) = \rho(x) - \sigma(x), \qquad
x \in\OOO, \thinspace 0<t<T.                 
\ee
Then we have
\bel{3.2}
\rho\ppp_t^2y(x,t) - \ddd (p(x)\nabla y(x,t))
= \ddd (f(x)\nabla R) - g\ppp_t^2R(x,t) \quad \mbox{in $\OOO\times (0,T)$},
\ee
and
\bel{3.3}
y(x,0) = \ppp_ty(x,0) = 0, \quad x \in \OOO, \quad y\vert_{\ppp\OOO}
= 0.                                     
\ee
We take the even extensions of the functions $R(x,t)$, $y(x,t)$ on
$t\in (-T,0)$.  For simplicity, we denote the extended functions by
the same notations $R(x,t), y(x,t)$.  Since $y \in
W^{4,\infty}(\OOO\times (0,T))$, $y(\cdot,0) = \ppp_ty(\cdot,0) = 0$
and $\ppp_t\nabla R(\cdot,0) = 0$
by $\ppp_tu(q,\sigma,a,h)(\cdot,0) = 0$ in $\OOO$, we see
that $(\ppp_t^3R)(\cdot,0) = (\ppp_t^3y)(\cdot,0) = 0$
in $\OOO$, and so $R\in W^{4,\infty}(Q)$,
$$
y \in W^{4,\infty}(Q)
$$
and
\bel{3.4}
\left\{ \begin{array}{rl}
& \rho\ppp_t^2y(x,t) - \ddd (p(x)\nabla y(x,t))
= \ddd (f(x)\nabla R) - g\ppp_t^2R(x,t) \quad \mbox{in $Q$},\\
&y(x,0) = \ppp_ty(x,0) = 0, \quad x \in \OOO, \\
&y = 0 \quad \mbox{on $\ppp\OOO \times (-T,T)$}.
\end{array}\right.                       
\ee

We set
\bel{3.5}
y_1 = y_1(a) = \ppp_ty(a), \quad y_2 = y_2(a) = \ppp_t^2y(a).
\ee
Henceforth we write $y_1$ and $y_2$ in place of $y_1(a)$ and
$y_2(a)$ when there is no fear of confusion.
Then
$$
\ppp_t^2R(x,0) = \ppp_t^2u(q,\sigma,a,h)(x,0)
= \frac{1}{\sigma}\ddd(q(x)\nabla u(q,\sigma,a,h))\vert_{t=0}
= \frac{\ddd (q\nabla a)}{\sigma}
$$
and $\ppp_ty_2(x,0) = \ppp_t^3y(x,0) = 0$ for $x \in \OOO$, because
we can differentiate the first equation in \eqref{3.4} and substitute
$t=0$ in terms of $y \in W^{4,\infty}(Q)$.  Hence we have
\bel{3.6}
\left\{ \begin{array}{rl}
& \rho\ppp_t^2y_1(x,t) - \ddd (p(x)\nabla y_1(x,t))
= \ddd (f(x)\nabla \ppp_tR) - g\ppp_t^3R =: G_1 \quad \mbox{in $Q$},\\
&y_1(x,0) = 0, \\
& \ppp_ty_1(x,0) = \frac{1}{\rho}\ddd(f\nabla a)
- g\frac{\ddd (q\nabla a)}{\rho\sigma}, \\
&y_1 = 0 \quad \mbox{on $\ppp\OOO \times (-T,T)$}
\end{array}\right.
\ee
and
\bel{3.7}
\left\{ \begin{array}{rl}
& \rho\ppp_t^2y_2(x,t) - \ddd (p(x)\nabla y_2(x,t))
= \ddd (f(x)\nabla \ppp_t^2R) - g\ppp_t^4R =: G_2 \quad \mbox{in $Q$},\\
&y_2(x,0) = \frac{1}{\rho}\ddd(f\nabla a)
- g\frac{\ddd (q\nabla a)}{\rho\sigma}, \\
&\ppp_ty_2(x,0) = 0, \quad x \in \OOO, \\
&y_2 = 0 \quad \mbox{on $\ppp\OOO \times (-T,T)$}.
\end{array}\right.
\ee

{\bf Second Step: weighted energy estimate and Carleman estimate.}

Let $k=1,2$. First, by multiplying the first equations in \eqref{3.6} and
\eqref{3.7} by $2\ppp_ty_k$, we can readily see
\bel{3.8}
\ppp_t(\rho\vert \ppp_ty_k\vert^2 + p\vert \nabla y_k\vert^2)
- \ddd (2p(\ppp_ty_k)\nabla y_k)
= 2(\ppp_ty_k)G_k \quad \mbox{in $Q$}.    
\ee
Multiplying \eqref{3.8} by $\chi(t)\weight$ and integrating by parts over
$\OOO \times (-T,0)$, we have
$$
\int^0_{-T} \int_{\OOO} \{\chi \weight
\ppp_t(\rho\vert \ppp_ty_k\vert^2)
+ \chi \weight \ppp_t(p\vert \nabla y_k\vert^2) \}dxdt   
$$
\bel{3.9}
- \int^0_{-T} \int_{\omega} \chi \weight \ddd (2p(\ppp_ty_k)\nabla y_k)
dxdt
= \int^0_{-T}\int_{\OOO} \chi \weight G_k 2(\ppp_ty_k) dxdt.
\ee
For $k=2$, by $y_2\vert_{\ppp\OOO} = 0$, $\chi(-T) = 0$ and the initial
condition of $y_2$, we have
\begin{align*}
& \mbox{[the left-hand side of \eqref{3.9}]} \\
= & \int_{\OOO} [\chi \weight \rho\vert \ppp_ty_2\vert^2]
^{t=0}_{t=-T} dx
- \int^0_{-T}\int_{\OOO} (\chi' + 2s\chi\ppp_t\va)\rho
\vert \ppp_ty_2\vert^2 \weight dxdt\\
+& \int_{\OOO} [\chi\weight p\vert \nabla y_2\vert^2]^{t=0}_{t=-T} dx
- \int^0_{-T}\int_{\OOO} (\chi' + 2s\chi\ppp_t\va)p\vert \nabla y_2\vert^2
\weight dxdt\\
+& \int^0_{-T}\int_{\OOO} 2s\chi(\nabla \va \cdot \nabla y_2)
2p(\ppp_ty_2)\weight dxdt\\
\ge &\int_{\OOO} p\vert \nabla y_2(x,0)\vert^2 e^{2s\va(x,0)} dx
- C\int_Q s\vert \nabla_{x,t}y_2\vert^2 \weight dxdt.
\end{align*}
Here we augmented the integral over $\OOO \times (-T,0)$ to $Q:=
\OOO \times (-T,T)$, and used $\vert \chi' + 2s \chi \ppp_t\va\vert \le Cs$ in
$Q$ and
$$
\vert 2s\chi (\nabla \va \cdot \nabla y_2)\ppp_ty_2\vert
\le Cs\vert \nabla y_2\vert \vert \ppp_ty_2\vert
\le Cs\vert \nabla_{x,t}y_2\vert^2 \quad \mbox{in $Q$}.
$$
Moreover
\bel{3.10}
\mbox{[the right-hand side of \eqref{3.9}]}
\le C\int_Q \vert G_2\vert^2 \weight dxdt
+ C\int_Q s\vert \ppp_ty_2\vert^2 \weight dxdt. 
\ee
Therefore \eqref{3.9} and \eqref{3.10} yield
\bel{3.11}
\int_{\OOO} \vert \nabla y_2(x,0)\vert^2 e^{2s\va(x,0)} dx
\le C\int_Q \vert G_2\vert^2 \weight dxdt
+ C\int_Q s\vert \nabla_{x,t}y_2\vert^2 \weight dxdt.    
\ee

Applying Lemma \ref{L2} to \eqref{3.7} and substituting it into \eqref{3.11},
we obtain
\bel{3.12}
\int_{\OOO} \vert \nabla y_2(x,0)\vert^2 e^{2s\va(x,0)} dx
\le C\int_Q \vert G_2\vert^2 \weight dxdt
+ Cs^3e^{2s(d_0-\ep_0)}\Vert y_2\Vert^2_{H^1(Q)}
+ CD_2^2                         
\ee
for $s \ge s_0$.  Here and henceforth we set
\bel{3.13}
D_k^2 := s^3e^{Cs}\Vert y_k\Vert^2_{H^1(-T,T;L^2(\omega))}, \quad
k=1,2.                               
\ee
For $k=1$, we can similarly argue to have
$$
\int_{\OOO} \vert y_2(x,0)\vert^2 e^{2s\va(x,0)} dx
= \int_{\OOO} \vert \ppp_ty_1(x,0)\vert^2 e^{2s\va(x,0)}dx
$$
\bel{3.14}
\le C\int_Q \vert G_1\vert^2 \weight dxdt
+ Cs^3e^{2s(d_0-\ep_0)}\Vert y_1\Vert^2_{H^1(Q)}
+ CD_1^2                         
\ee
Hence \eqref{3.12} and \eqref{3.14} imply
\bel{3.15}
\int_{\OOO} (\vert \nabla y_2(x,0)\vert^2 + \vert y_2(x,0)\vert^2)
e^{2s\va(x,0)} dx                      
\ee
\begin{align*}
\le& C\int_Q (\vert G_1\vert^2 + \vert G_2\vert^2) \weight dxdt
+ Cs^3e^{2s(d_0-\ep_0)}(\Vert y_1\Vert^2_{H^1(Q)}
+ \Vert y_2\Vert^2_{H^1(Q)})\\
+& C(D_1^2+D_2^2)
\end{align*}
for $s \ge s_0$.

{\bf Third Step: Energy estimate for $\Vert y_1\Vert^2_{H^1(Q)}$
and  $\Vert y_2\Vert^2_{H^1(Q)}$.}

Applying a usual energy estimate to \eqref{3.6} and \eqref{3.7}, in terms of the
Poincar\'e inequality, we have
\begin{align*}
& \int_{\OOO} (\vert \nabla_{x,t}y_k(x,t)\vert^2
+ \vert y_k(x,t)\vert^2) dx \\
\le& C \int_{\OOO} (\vert \nabla_{x,t} y_k(x,0)\vert^2+ \vert y_k(x,0)\vert^2) dx
+ C\int^T_{-T}\int_{\OOO} \vert G_k\vert^2  dxdt,\quad k=1, 2,
\end{align*}

for $-T \le t \le T$.  Consequently
\bel{3.16}
 \Vert y_k\Vert^2_{H^1(Q)}
\le C\int_{\OOO} (\vert \nabla_{x,t} y_k(x,0)\vert^2
+ \vert y_k(x,0)\vert^2) dx
+ C\int_Q \vert G_k\vert^2  dxdt,\quad k=1, 2.   
\ee

Substituting \eqref{3.16} in \eqref{3.15} and using $e^{2s\va} \ge 1$, we obtain
\begin{align*}
& \int_{\OOO} (\vert \nabla y_2(x,0)\vert^2 + \vert y_2(x,0) \vert^2)
e^{2s\va(x,0)} dx \\
\le& C\int_Q (\vert G_1\vert^2 + \vert G_2\vert^2) \weight dxdt
+ Cs^3e^{2s(d_0-\ep_0)}\int_{\OOO} (\vert \nabla y_2(x,0)\vert^2
+ \vert y_2(x,0) \vert^2) dx \\
+ & Cs^3e^{2s(d_0-\ep_0)}\int_Q (\vert G_1\vert^2 + \vert G_2\vert^2) dxdt
+ C(D_1^2 + D_2^2),
\end{align*}
that is,
\begin{align*}
& \int_{\OOO} (\vert \nabla y_2(x,0)\vert^2 + \vert y_2(x,0) \vert^2)
e^{2s\va(x,0)} (1-Cs^3e^{2s(d_0-\ep_0-\va(x,0)}) dx \\
\le& Cs^3e^{2s(d_0-\ep_0)}\int_Q (\vert G_1\vert^2
+ \vert G_2\vert^2) dxdt
+ C\int_Q (\vert G_1\vert^2 + \vert G_2\vert^2) \weight dxdt\\
+ & C(D_1^2 + D_2^2),
\end{align*}
By \eqref{2.8}, choosing $s>0$ sufficiently large, we have
$$
1-Cs^3e^{2s(d_0-\ep_0-\va(x,0)} \ge 1 - Cs^3e^{-4\ep_0 s}
\ge \frac{1}{2}.
$$
Hence
\begin{align*}
& \int_{\OOO} (\vert \nabla y_2(x,0)\vert^2 + \vert y_2(x,0) \vert^2)
e^{2s\va(x,0)} dx \\
\le& Cs^3e^{2s(d_0-\ep_0)}\int_Q (\vert G_1\vert^2
+ \vert G_2\vert^2) dxdt
+ C\int_Q (\vert G_1\vert^2 + \vert G_2\vert^2) \weight dxdt
+ C(D_1^2 + D_2^2)
\end{align*}
for all large $s>0$.  By the definitions of $G_1$ and $G_2$ in
\eqref{3.6} and \eqref{3.7}, we see that
$$
\sum_{k=1}^2 \vert G_k\vert^2
\le C(\vert \nabla f\vert^2 + \vert f\vert^2 + \vert g\vert^2)
\quad \mbox{in $Q$}.
$$
Consequently, recalling \eqref{3.5}: $y_1 = y_1(a)$ and $y_2 = y_2(a)$,
we obtain
\bel{3.17}
\int_{\OOO} \vert \nabla y_2(a)(x,0)\vert^2e^{2s\va(x,0)} dx
\ee
\begin{align*}
\le& Cs^3e^{2s(d_0-\ep_0)}\int_Q (\vert \nabla f\vert^2
+ \vert f\vert^2 + \vert g\vert^2) dxdt
+ C\int_Q (\vert \nabla f\vert^2 + \vert f\vert^2
+ \vert g\vert^2) \weight dxdt\\
+ &C(D_1^2 + D_2^2).
\end{align*}
Substituting \eqref{3.16} in \eqref{3.14},  we can similarly argue to have
\bel{3.18}
\int_{\OOO} \vert y_2(a)(x,0) \vert^2 e^{2s\va(x,0)} dx         
\ee
\begin{align*}
\le& Cs^3e^{2s(d_0-\ep_0)}\int_Q (\vert \nabla f\vert^2
+ \vert f\vert^2 + \vert g\vert^2) dxdt
+ C\int_Q (\vert \nabla f\vert^2 + \vert f\vert^2
+ \vert g\vert^2) \weight dxdt\\
+ &CD_1^2
\end{align*}
for all large $s>0$.

Setting $a=a_1, a_2$, by the initial condition in \eqref{3.7}, we see
\bel{3.19}
\rho y_2(a_{\ell})(x,0)
= \ddd (f\nabla a_{\ell}) - \frac{\ddd (q\nabla a_{\ell})}{\sigma}g,
\quad \ell = 1,2.                                  
\ee
Then, eliminating $g$ in the two equations in \eqref{3.19}, we obtain
\begin{align*}
& (\ddd (q\nabla a_2)\nabla a_1 - \ddd (q\nabla a_1)\nabla a_2)\cdot \nabla f
+ ((\ddd(q\nabla a_2)\Delta a_1 - (\ddd(q\nabla a_1)\Delta a_2)f\\
=& \rho\ddd(q\nabla a_2)y_2(a_1)(x,0)
-  \rho\ddd(q\nabla a_1)y_2(a_2)(x,0) \quad \mbox{in $Q$}.
\end{align*}
Applying \eqref{2.15} in Lemma \ref{L3} to this first-order equation in $f$, by the second condition in (1.9), we have
\bel{3.20}
s^2\int_{\OOO} (\vert \nabla f\vert^2 + \vert f\vert^2) e^{2s\va(x,0)} dx
\ee
\begin{align*}
\le & \int_{\OOO} \vert \ddd(q\nabla a_2)y_2(a_1)(x,0)
-  \ddd(q\nabla a_1)y_2(a_2)(x,0)\vert^2 e^{2s\va(x,0)} dx\\
+ &C\int_{\OOO} \vert \nabla(\ddd(q\nabla a_2)y_2(a_1)(x,0)
-  \ddd(q\nabla a_1)y_2(a_2)(x,0))\vert^2 e^{2s\va(x,0)} dx\\
\le& C\int_{\OOO} \left(
\sum_{\ell=1}^2 (\vert \nabla y_2(a_{\ell})(x,0)\vert^2
+ \vert y_2(a_{\ell})(x,0)\vert^2\right) e^{2s\va(x,0)} dx.
\end{align*}
Moreover, assuming that the first condition in \eqref{1.8} holds with $\ell=1$
for example, we have
$$
g = \frac{\sigma}{\ddd (q\nabla a_1)}(\ddd (f\nabla a_1)
- \rho y_2(a_1)(x,0))   \quad \mbox{on $\ooo{\OOO}$},
$$
and so
$$
\vert g(x)\vert \le C(\vert \nabla f(x)\vert + \vert f(x)\vert
+ \vert y_2(a_1)(x,0)\vert), \quad x \in \ooo{\OOO}.
$$
Hence, applying \eqref{3.20} and \eqref{3.17}-\eqref{3.18} for $y_2(a_1)(x,0)$ and
$y_2(a_2)(x,0)$, we obtain
\bel{3.21}
\int_{\OOO} (\vert \nabla f\vert^2 + \vert f\vert^2 + \vert g\vert^2)
e^{2s\va(x,0)} dx                         
\ee
\begin{align*}
\le & Cs^3e^{2s(d_0-\ep_0)}
\int_{\OOO} (\vert \nabla f\vert^2 + \vert f\vert^2 + \vert g\vert^2) dx\\
+& C\int_Q (\vert \nabla f\vert^2 + \vert f\vert^2 + \vert g\vert^2)
\weight dxdt + C\www{D}^2.
\end{align*}
Here we used $\vert y_k(x,-t)\vert = \vert y_k(x,t)\vert$, $k=1,2$ which is
seen by the even extension of $y(\cdot,t)$ in $t$, and recall \eqref{3.13}, and
we set
\bel{3.22}
\www{D}^2 := \sum_{\ell=1}^2 \Vert u(p,\rho,a_{\ell},h_{\ell})
- u(q,\sigma,a_{\ell},h_{\ell})\Vert^2_{H^3(0,T;L^2(\omega))}.
\ee
We will estimate the second term on the right-hand side of \eqref{3.21} as follows.
\begin{align*}
& \int_Q (\vert \nabla f\vert^2 + \vert f\vert^2 + \vert g\vert^2)
\weight dxdt\\
=& \int_{\OOO} (\vert \nabla f\vert^2 + \vert f\vert^2 + \vert g\vert^2)
e^{2s\va(x,0)} \left( \int^T_{-T}
e^{2s(\va(x,t) - \va(x,0))} dt \right) dx.
\end{align*}
Since
\begin{align*}
& \va(x,t) - \va(x,0) = e^{\lambda\vert x-x_0\vert^2}
(e^{-\lambda \beta t^2} - 1)\\
\le& -e^{\lambda\min_{x\in\ooo{\OOO}}\vert x-x_0\vert^2}
(1-e^{-\lambda\beta t^2})
\le -C_0(1-e^{-\lambda\beta t^2})  \quad \mbox{in $Q$},
\end{align*}
we have
$$
\int^T_{-T} e^{2s(\va(x,t) - \va(x,0))} dt
\le \int^T_{-T} \exp(-2sC_0(1-e^{-\lambda\beta t^2})) dt
= o(1)
$$
as $s \to \infty$, where we used the Lebesgue convergence theorem.
Therefore
$$
\int_Q (\vert \nabla f\vert^2 + \vert f\vert^2 + \vert g\vert^2)
\weight dxdt
\le o(1) \int_{\OOO} (\vert \nabla f\vert^2 + \vert f\vert^2 + \vert g\vert^2)
e^{2s\va(x,0)} dx
$$
as $s \to \infty$, and choosing $s>0$ sufficiently large, we can
absorb the second term on the right-hand side of \eqref{3.21} into the left-hand side.
By \eqref{2.8}, we have $e^{2s\va(x,0)} \ge e^{2s(d_0+\ep_0)}$,
so that from \eqref{3.21} we obtain
\begin{align*}
&e^{2s(d_0+\ep_0)}
\int_{\OOO} (\vert \nabla f\vert^2 + \vert f\vert^2 + \vert g\vert^2) dx\\
\le &Cs^3e^{2s(d_0-\ep_0)}
\int_{\OOO} (\vert \nabla f\vert^2 + \vert f\vert^2 + \vert g\vert^2) dx
+ C\www{D}^2
\end{align*}
for all large $s>0$.  For large $s>0$, we see that
$e^{2s(d_0+\ep_0)} - Cs^3e^{2s(d_0-\ep_0)} > 0$.  Hence fixing such
$s>0$, we reach
\bel{3.23}
\int_{\OOO} (\vert \nabla f\vert^2 + \vert f\vert^2 + \vert g\vert^2)
e^{2s\va(x,0)} dx \le C\www{D}^2.                 
\ee
By the definition \eqref{3.22} of $\www{D}^2$, the proof of Theorem \ref{T1} is
completed.

\subsection{Proof of Theorem \ref{T2}.}
Again we set
\begin{align*}
& \rho y_2(a_{\ell})(x,0)
= \ddd (f\nabla a_{\ell}) - \frac{\ddd (q\nabla a_{\ell})}{\sigma}g\\
= &\sum_{k=1}^n (\ppp_ka_{\ell})\ppp_kf + (\Delta a_{\ell})f
- \frac{\ddd (q\nabla a_{\ell})}{\sigma}g, \quad \ell=1, ..., n+1.
\end{align*}
that is,
\bel{3.24}
\sum_{k=1}^n (\ppp_ka_{\ell})\ppp_kf - \frac{\ddd (q\nabla a_{\ell})}{\sigma}g
= \rho y_2(a_{\ell})(x,0) - (\Delta a_{\ell})f, \quad \ell=1, ..., n+1.
\ee
We rewrite \eqref{3.24} as a linear system with respect to $(n+1)$ unknowns
$\ppp_1f$, ..., $\ppp_nf$, $g$:
\begin{align*}
& \left( \begin{array}{cccc}
\ppp_1a_1 & \cdots & \ppp_na_1
& -\frac{1}{\sigma}\sum_{k=1}^n (\ppp_kq)\ppp_ka_1
- \frac{q\Delta a_1}{\sigma}\\
\vdots & \vdots & \vdots &\vdots \\
\ppp_1a_{n+1} & \cdots & \ppp_na_{n+1}
& -\frac{1}{\sigma}\sum_{k=1}^n (\ppp_kq)\ppp_ka_{n+1}
- \frac{q\Delta a_{n+1}}{\sigma}\\
\end{array}\right)
\left( \begin{array}{cc}
\ppp_1 f \\
\vdots \\
\ppp_nf\\
g  \\
\end{array}\right) \\
=& \left( \begin{array}{cc}
\rho y_2(a_1)(x,0) - (\Delta a_1)f\\
\vdots\\
\rho y_2(a_{n+1})(x,0) - (\Delta a_{n+1})f\\
\end{array}\right).
\end{align*}
In the coefficient matrix, multiplying the $j$-th column by
$\frac{1}{\sigma}\ppp_jq$, $j=1, 2, ..., n$ and adding them to
the $(n+1)$-th column, we obtain
\begin{align*}
& \mbox{[the determinant of the coefficient matrix]}\\
=& \mbox{det}\thinspace
   \left( \begin{array}{cccc}
\ppp_1a_1 & \cdots & \ppp_na_1 & - \frac{q\Delta a_1}{\sigma}\\
\vdots & \vdots & \vdots &\vdots \\
\ppp_1a_{n+1} & \cdots & \ppp_na_{n+1} & - \frac{q\Delta a_{n+1}}{\sigma}\\
\end{array}\right)\\
=& -\frac{q}{\sigma} \mbox{det}\thinspace
   \left( \begin{array}{cccc}
\ppp_1a_1 & \cdots & \ppp_na_1 & \Delta a_1\\
\vdots & \vdots & \vdots &\vdots \\
\ppp_1a_{n+1} & \cdots & \ppp_na_{n+1} & \Delta a_{n+1}\\
\end{array}\right) \quad \mbox{on $\ooo{\OOO}$}.
\end{align*}
Therefore by the assumption \eqref{1.13}, there exists a constant $C>0$,
independent of choices of $(p,\rho)$ and $(q,\sigma)$, such that
$$
\vert \nabla f(x)\vert^2 + \vert g(x)\vert^2
\le C\left(\sum_{\ell=1}^{n+1} \vert \rho y_2(a_{\ell})(x,0)\vert^2
+ \vert f(x)\vert^2\right), \quad x \in \ooo{\OOO},
$$
and so
\bel{3.25}
\int_{\OOO} (\vert \nabla f\vert^2 + \vert g\vert^2)
e^{2s\va(x,0)} dx
\le C\int_{\OOO} \sum_{\ell=1}^{n+1} \vert y_2(a_{\ell})(x,0)\vert^2
e^{2s\va(x,0)} dx
+ \int_{\OOO} \vert f(x)\vert^2 e^{2s\va(x,0)} dx.     
\ee

We consider a first-order partial differential operator:
\bel{3.26}
(Q_0f)(x) = (x-x_0)\cdot \nabla f(x), \quad x \in \OOO. 
\ee
By $x_0 \not\in \ooo{\OOO}$, the condition (2.13) is satisfied, and
\eqref{2.14} in Lemma \ref{L3} yields
\begin{align*}
& s^2\int_{\OOO} \vert f(x)\vert^2 e^{2s\va(x,0)} dx
\le C\int_{\OOO} \vert ((x-x_0)\cdot \nabla f(x)\vert^2 e^{2s\va(x,0)} dx\\
\le &  C\int_{\OOO} \vert \nabla f(x)\vert^2 e^{2s\va(x,0)} dx
\end{align*}
for all large $s>0$.  Therefore
$$
\int_{\OOO} \vert f(x)\vert^2 e^{2s\va(x,0)} dx
\le \frac{C}{s^2}\int_{\OOO} \vert \nabla f(x)\vert^2 e^{2s\va(x,0)} dx
$$
for all large $s>0$.
Substituting this inequality into the second term on the right-hand side of
\eqref{3.25} and absorbing into the left-hand side by choosing $s>0$ large,
in terms of \eqref{3.18} with $y_2(a_{\ell})$, $\ell=1,2,..., n+1$,
\begin{align*}
& \int_{\OOO} (\vert \nabla f\vert^2 + \vert f\vert^2 + \vert g\vert^2)
e^{2s\va(x,0)} dxdt
\le C\int_{\OOO} \sum_{\ell=1}^{n+1} \vert y_2(a_{\ell})(x,0)\vert^2
e^{2s\va(x,0)} dx \\
\le& Cs^3e^{2s(d_0-\ep_0)}
\int_{\OOO} (\vert \nabla f\vert^2 + \vert f\vert^2 + \vert g\vert^2)
e^{2s\va(x,0)} dxdt\\
+ &C\int_Q (\vert \nabla f\vert^2 + \vert f\vert^2 + \vert g\vert^2)
\weight dxdt
+ C\sum_{\ell=1}^{n+1} s^3e^{Cs}\Vert y_1(a_{\ell})\Vert^2
_{H^1(-T,T;L^2(\omega))}
\end{align*}
for all large $s>0$.  Similarly to \eqref{3.23}, we can absorb the first and the
second terms on the right-hand side into the left-hand side, so that
we can complete the proof of Theorem \ref{T2}.

 \begin{figure}[tbp]
 \begin{center}
 \begin{tabular}{cc}
{\includegraphics[scale=0.25, clip=]{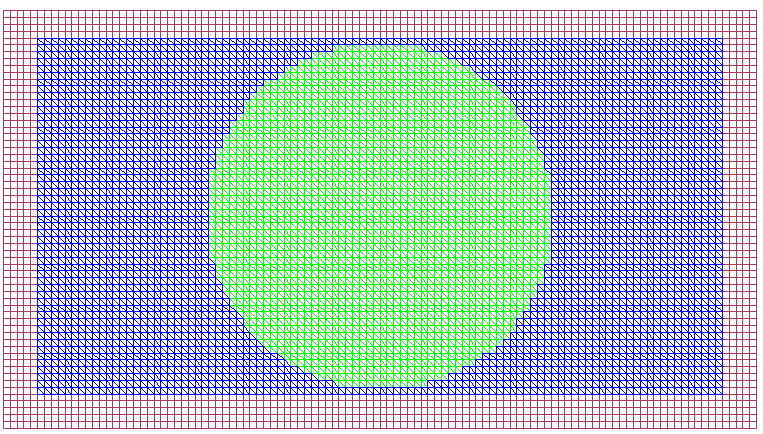}} &
\begin{tikzpicture}[scale=0.35,x=0.74cm,y=0cm]
  \node[anchor=north west,inner sep=0pt] at (0,0) {\includegraphics[scale=0.25, clip=]{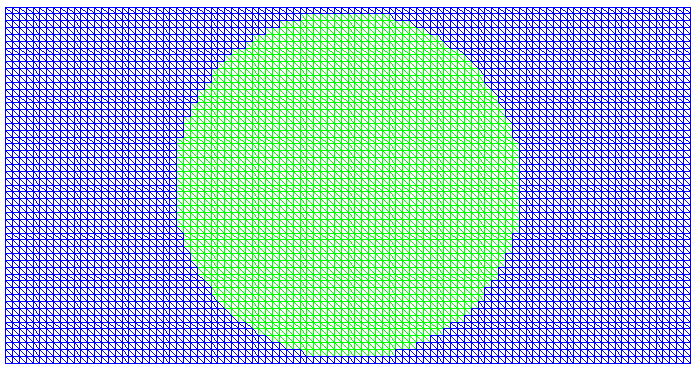}};
\draw (20,-75pt) node[anchor=north] {$\Omega_2$};
\draw (12,-75pt) node[anchor= north] {$\Omega_1$};
\end{tikzpicture}
\\
a) $\Omega$ & b)   $\Omega_{FEM} = \Omega_1 \cup \Omega_2$
 \end{tabular}
 \end{center}
 \caption{ \emph{a) Computational  mesh  used in the domain
     decomposition of the domain $\Omega =
\Omega_{FEM} \cup \Omega_{FDM}$. b) The finite element
     mesh in $\Omega_{FEM} = \Omega_1 \cup \Omega_2$.}}
 \label{fig:0_1}
 \end{figure}

\begin{figure}
 \begin{center}
   \begin{tabular}{|c|c|c|c|}
       \hline
     \multicolumn{4}{|c|}
                 { Exact $\rho(x)$}
                 \\
                 \hline
                  Test 1 &   Test 2 &   Test 3 &   Test 4 \\
                 \hline
{\includegraphics[scale=0.19, clip=true,]{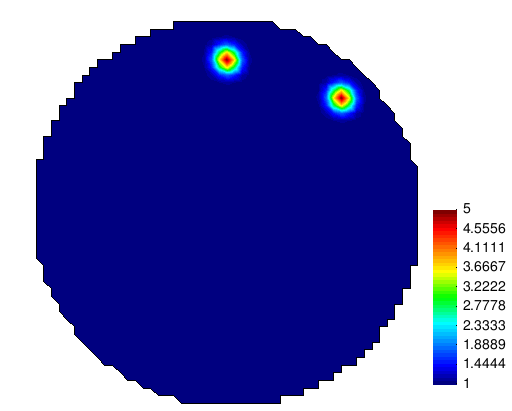}} &
{\includegraphics[scale=0.19,  clip=true,]{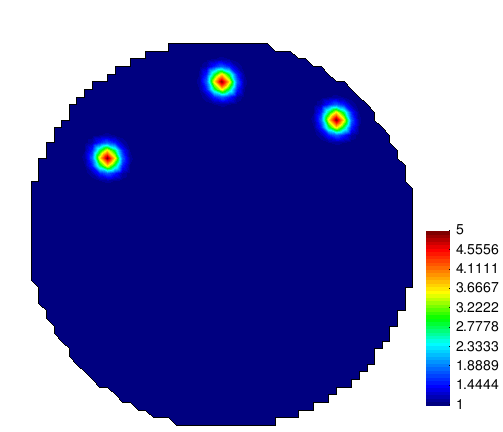}} &
 {\includegraphics[scale=0.19, clip=true,]{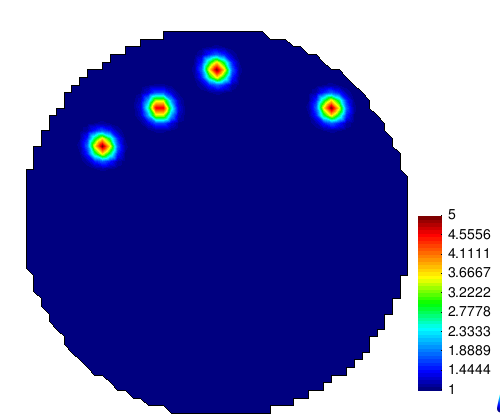}} &
 {\includegraphics[scale=0.19,  clip=true,]{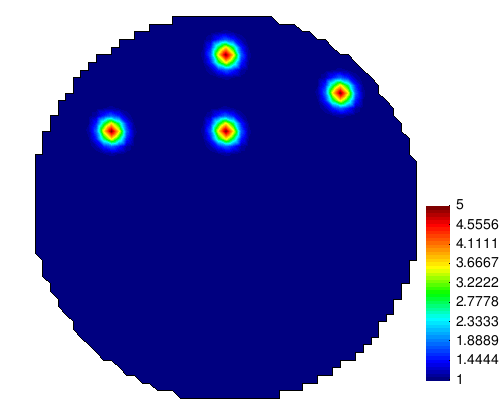}}
 \\
   \hline
     \multicolumn{4}{|c|}
                 {Exact $p(x)$}
                 \\
                 \hline
    Test 1 &   Test 2 &   Test 3 &   Test 4 \\
                 \hline
   {\includegraphics[scale=0.19, clip=true,]{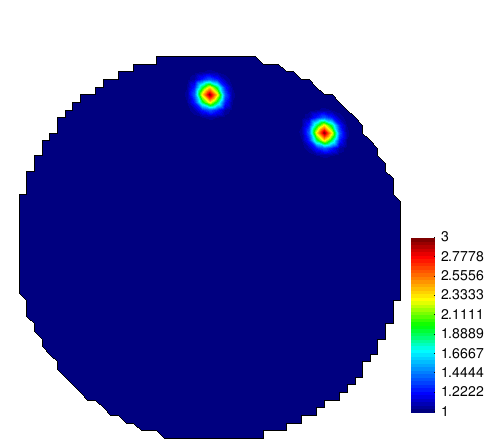}} &
{\includegraphics[scale=0.19,  clip=true,]{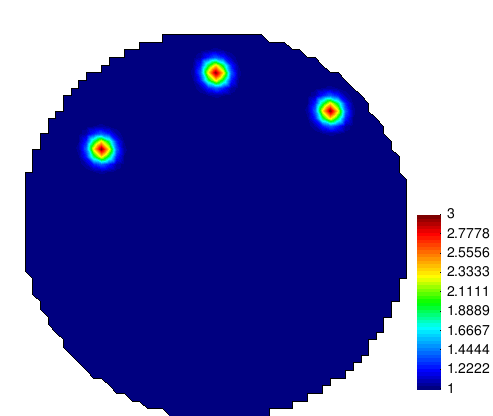}} &
 {\includegraphics[scale=0.19, clip=true,]{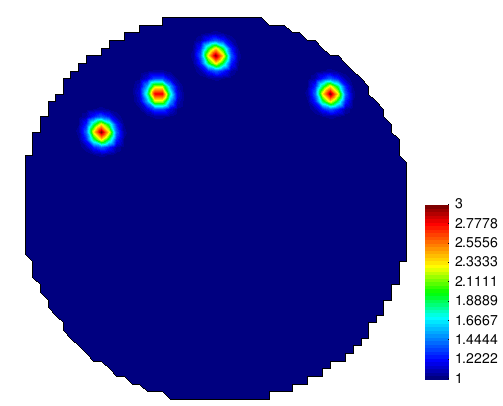}} &
 {\includegraphics[scale=0.19,  clip=true,]{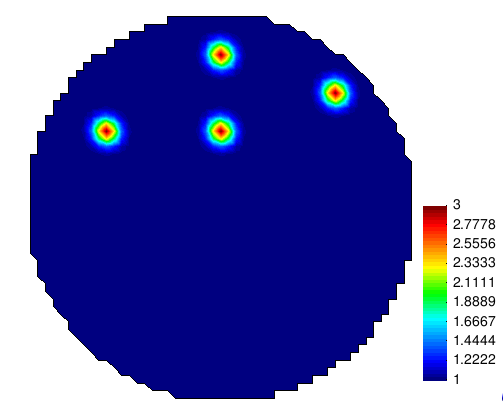}} \\
 \hline
\end{tabular}
 \end{center}
 \caption{ \protect\small \emph{Exact Gaussian functions $\rho(x)$ and
     $p(x)$ in $\Omega_1$ in different tests. }}
 \label{fig:exact_gaussians}
 \end{figure}

 \begin{figure}
 \begin{center}
   \begin{tabular}{|c|c|c|c|}
     \hline
     \multicolumn{4}{|c|}
                 {Test 1}
                 \\
   \hline
  $\rho(x), \delta=3\%$  & $ \rho(x), \delta=10\%$  &  $p(x), \delta= 3\%$ &  $p(x), \delta= 10\%$ \\
\hline
    {\includegraphics[scale=0.2, clip=true,]{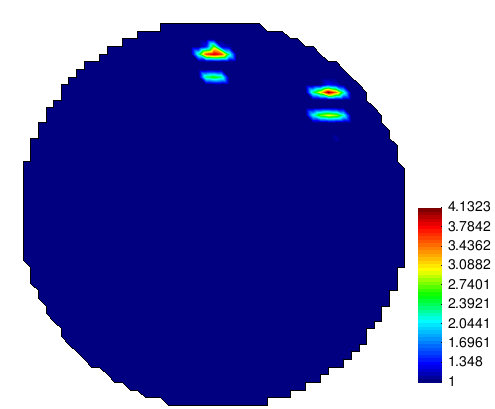}} &
    {\includegraphics[scale=0.2, clip=true,]{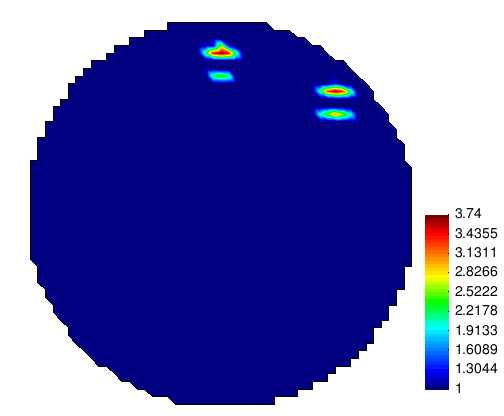}} &
    {\includegraphics[scale=0.2,  clip=true,]{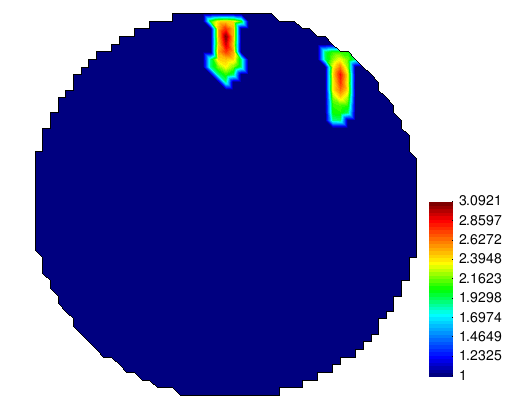}} &
     {\includegraphics[scale=0.2,  clip=true,]{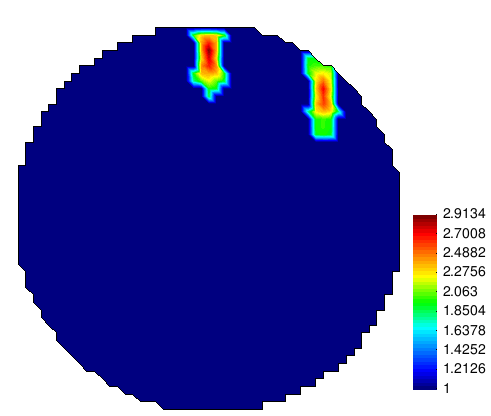}}
   \\
   \hline
    \multicolumn{4}{|c|}
                 {Test 2}
                 \\
                 \hline
                     {\includegraphics[scale=0.2, clip=true,]{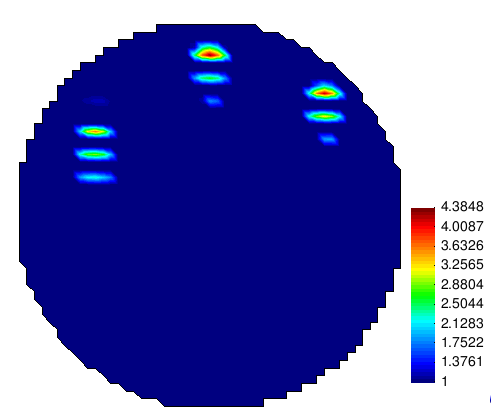}} &
                     {\includegraphics[scale=0.2, clip=true,]{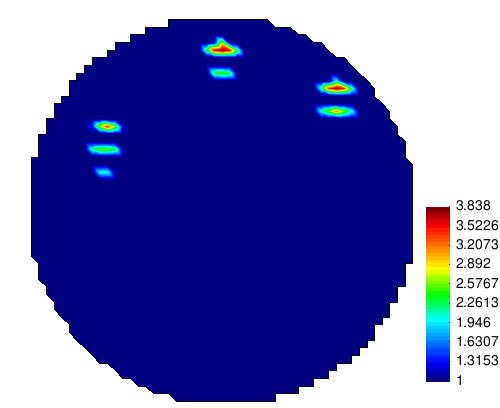}} &
                     {\includegraphics[scale=0.2,  clip=true,]{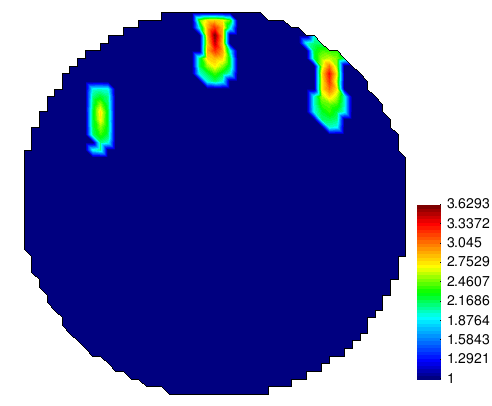}} &
                     {\includegraphics[scale=0.2,  clip=true,]{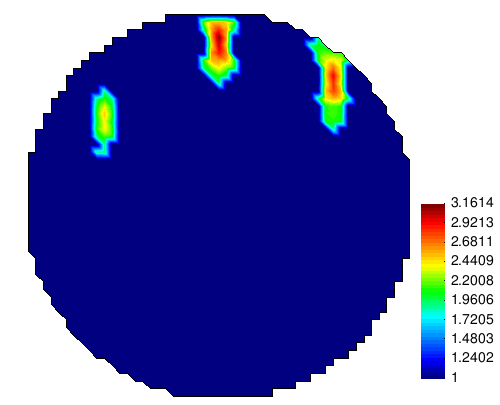}}
   \\
   \hline
   \multicolumn{4}{|c|}
                 {Test 3}
                 \\
                 \hline
      {\includegraphics[scale=0.2, clip=true,]{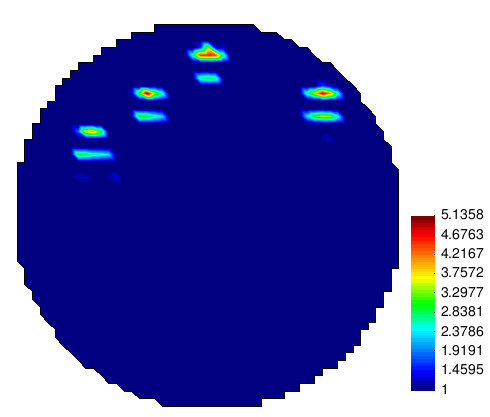}} &
      {\includegraphics[scale=0.2, clip=true,]{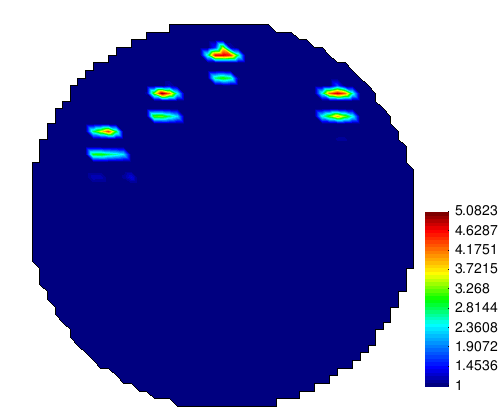}} &
      {\includegraphics[scale=0.2,  clip=true,]{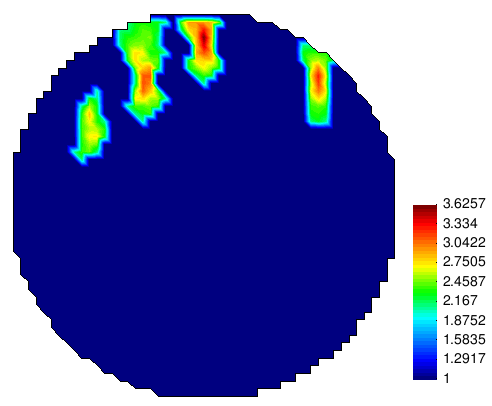}} &
      {\includegraphics[scale=0.2,  clip=true,]{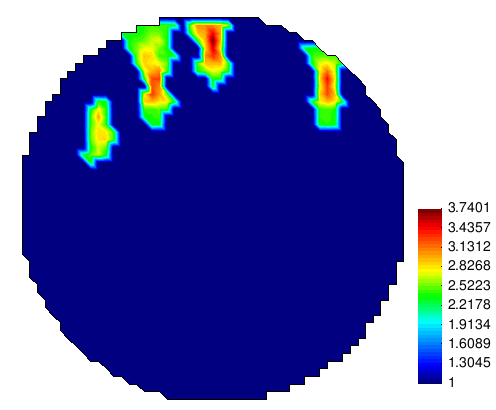}}
   \\
   \hline
    \multicolumn{4}{|c|}
                 {Test 4}
                 \\
                 \hline
      {\includegraphics[scale=0.2, clip=true,]{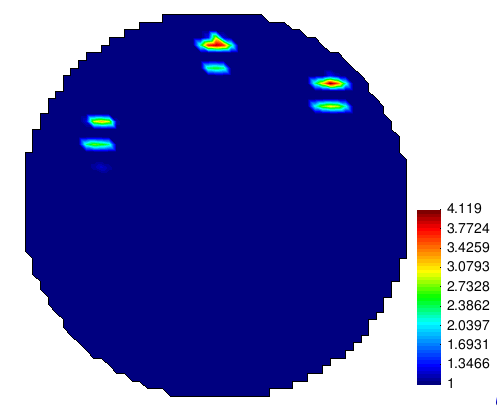}} &
 {\includegraphics[scale=0.2, clip=true,]{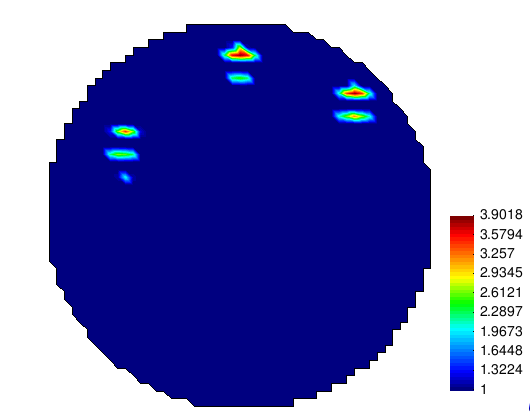}}     &
 {\includegraphics[scale=0.2,  clip=true,]{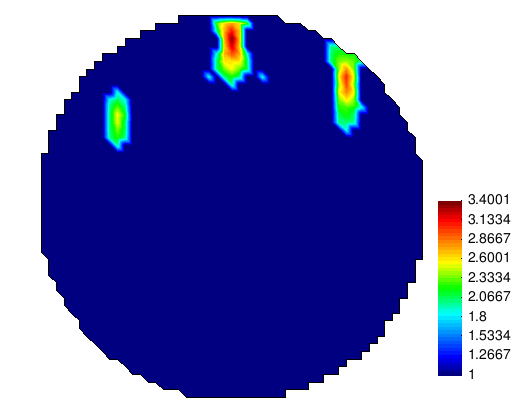}}     &
   {\includegraphics[scale=0.2,  clip=true,]{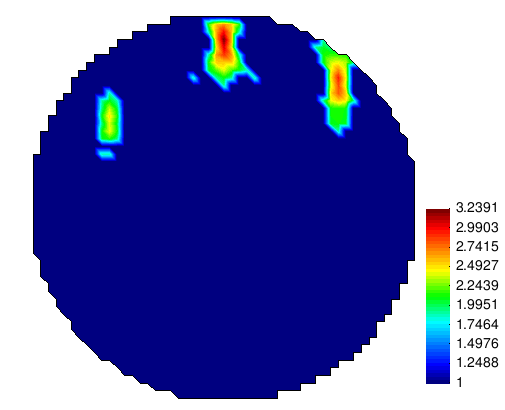}}
   \\
   \hline
\end{tabular}
 \end{center}
 \caption{ \protect\small \emph{  Reconstructions
   obtained in Tests 1-4 on a coarse mesh for different noise levels  $\delta$ in data. }}
 \label{fig:rec_test}
 \end{figure}

\section{Numerical Studies}\label{S4}
\label{sec:Numer-Simul}

In this section, we present numerical simulations of the
reconstruction of two unknown functions $\rho(x)$ and $p(x)$ of the
equation \eqref{11} using the domain decomposition method of
\cite{BAbsorb}.

To do that we decompose the computational domain $\Omega$ into two
subregions $\Omega_{FEM}$ and $\Omega_{FDM}$ such that $\Omega =
\Omega_{FEM} \cup \Omega_{FDM}$ with two layers of structured
overlapping nodes between these domains, see Figure \ref{fig:0_1} and
Figure 2 of \cite{B} for details about communication between
$\Omega_{FEM}$ and $\Omega_{FDM}$.  We will apply in our computations
the finite element method (FEM) in $\Omega_{FEM}$ and the finite
difference method (FDM) in $\Omega_{FDM}$.  We also decompose the
domain $\Omega_{FEM}$ into two different domains $\Omega_1, \Omega_2$
such that $\Omega_{FEM} = \Omega_1 \cup \Omega_2$ which are
intersecting only by their boundaries, see Figure \ref{fig:0_1}.
We use the domain decomposition approach in our computations since it
 is   efficiently implemented in
the high performance software package WavES \cite{waves} using C++ and
PETSc  \cite{petsc}.
For
further details about construction of $\Omega_{FDM}$ and
$\Omega_{FEM}$ domains as well as the domain decomposition method we
refer to \cite{BAbsorb}.

The
boundary $\partial \Omega$ of the domain $\Omega$ is such that $\partial
\Omega =\partial _{1} \Omega \cup \partial _{2} \Omega \cup \partial
_{3} \Omega$ where $\partial _{1} \Omega$ and $\partial _{2} \Omega$
are, respectively, top and bottom parts of $\Omega$, and $\partial
_{3} \Omega$ is the union of left and right sides of this
domain. We will collect time-dependent observations $\Gamma_1 :=
\partial_1 \Omega \times (0,T)$ at the backscattering side $\partial_1
\Omega$ of $\Omega$.  We also define $\Gamma_{1,1} := \partial_1
\Omega \times (0,t_1]$, $\Gamma_{1,2} := \partial_1 \Omega \times
(t_1,T)$, $\Gamma_2 := \partial_2 \Omega \times (0, T)$ and $\Gamma_3
:= \partial_3 \Omega \times (0, T)$.

We have used the following
 model problem in  all computations:
\begin{equation}\label{model1}
\begin{split}
\rho(x)\partial_t^2 u(x,t) - \ddd ( (p(x)\nabla u(x,t)) &= 0~ \mbox{in}~~ \Omega_T, \\
  u(x,0) = a(x), ~~~u_t(x,0) &= 0~ \mbox{in}~~ \Omega,     \\
\partial _{n} u& = f(t) ~\mbox{on}~ \Gamma_{1,1},
\\
\partial _{n}  u& =-\partial _{t} u ~\mbox{on}~ \Gamma_{1,2},
\\
\partial _{n} u& =-\partial _{t} u~\mbox{on}~ \Gamma_2, \\
\partial _{n} u& =0~\mbox{on}~ \Gamma_3.\\
\end{split}
\end{equation}
 In (\ref{model1}) the function $f(t)$  represents the single direction of a
 plane wave which is initialized at $\partial_1 \Omega$ in time
 $t=[0,2.0]$ and is defined as
 \begin{equation}\label{f}
 \begin{split}
 f(t) =\left\{
 \begin{array}{ll}
 \sin \left( \omega_f t \right) ,\qquad &\text{ if }t\in \left( 0,\frac{2\pi }{\omega_f }
 \right) , \\
 0,&\text{ if } t>\frac{2\pi }{\omega_f }.
 \end{array}
 \right.
 \end{split}
 \end{equation}
We  initialize initial condition $a(x)$ at the boundary $\partial_1 \Omega$ as
\begin{equation}\label{initcond}
\begin{split}
u(x,0) &= f_0(x)={\rm e}^{-(x_1^2 + x_2^2 + x_3^3)}  \cdot \cos  t|_{t=0} = {\rm e}^{-(x_1^2 + x_2^2 + x_3^3)}.
\end{split}
\end{equation}

We assume that both functions $\rho(x)=p(x)=1$ are known inside
$\Omega_{FDM} \cup \Omega_2$.  The goal of our numerical tests is to
reconstruct simultaneously two smooth functions $\rho(x), p(x)$ of the
domain $\Omega_{1}$ of Figure \ref{fig:0_1}.  The main feature of
these functions is that they model inclusions of a very small sizes
what can be of practical interest in real-life applications.


 \begin{figure}
 \begin{center}
   \begin{tabular}{|c|c|c|c|}
     \hline
     \multicolumn{4}{|c|}
     {\begin{tabular}{cc}
         adaptively  refined meshes& zoomed \\
         \hline
         \\
         {\includegraphics[scale=0.27, clip=true,]{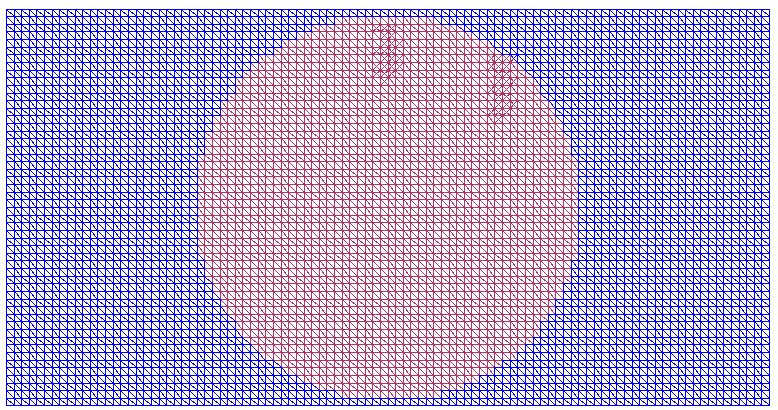}} &
         {\includegraphics[scale=0.27, trim = 6.0cm 0.0cm 0.0cm 0.0cm, clip=true,]{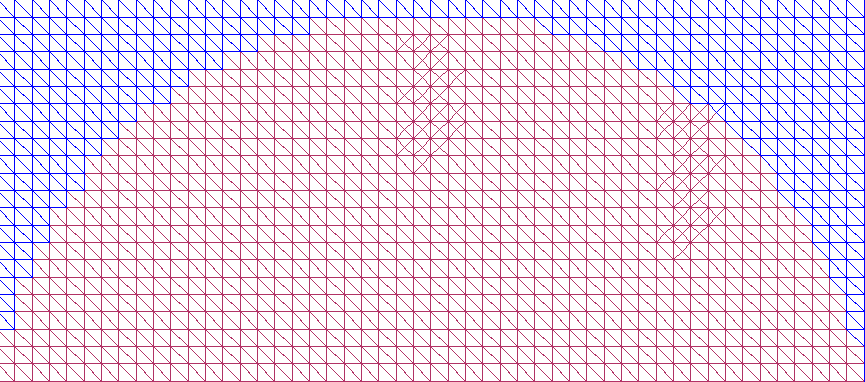}}    \\
once refined mesh &  \\ 
   {\includegraphics[scale=0.27, clip=true,]{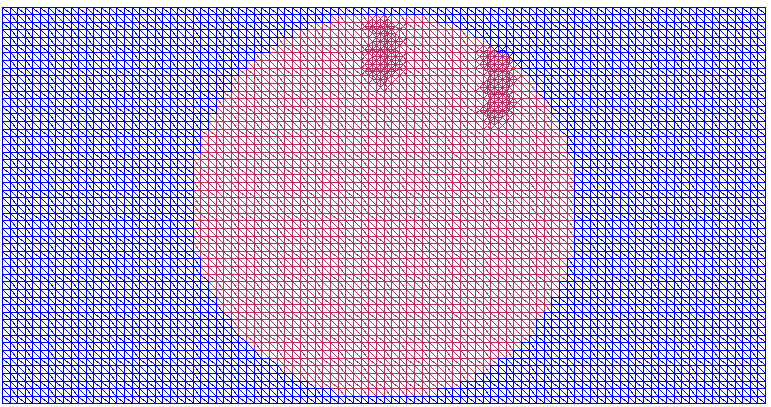}} &
   {\includegraphics[scale=0.27, trim = 6.0cm 0.0cm 0.0cm 0.0cm, clip=true,]{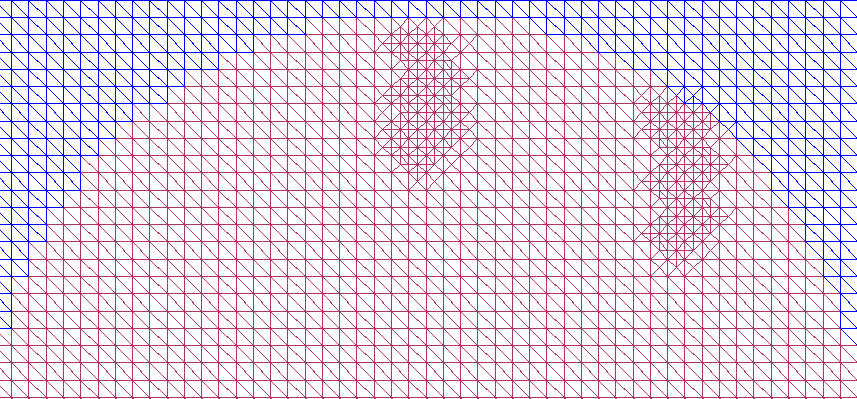}} \\
twice refined mesh &   \\
   {\includegraphics[scale=0.27, clip=true,]{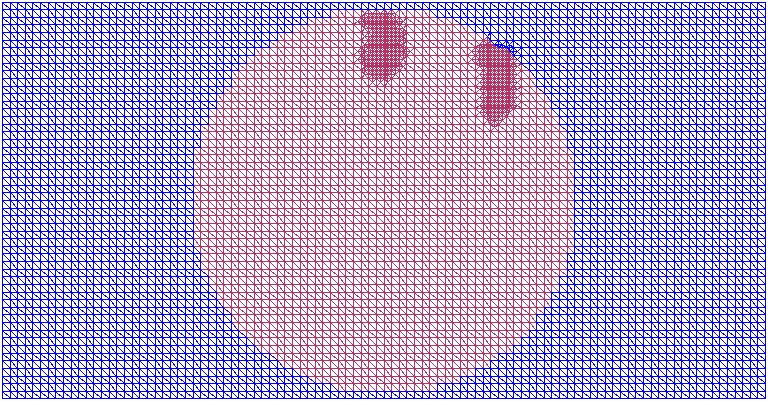}} &
   {\includegraphics[scale=0.27, trim = 6.0cm 0.0cm 0.0cm 0.0cm, clip=true,]{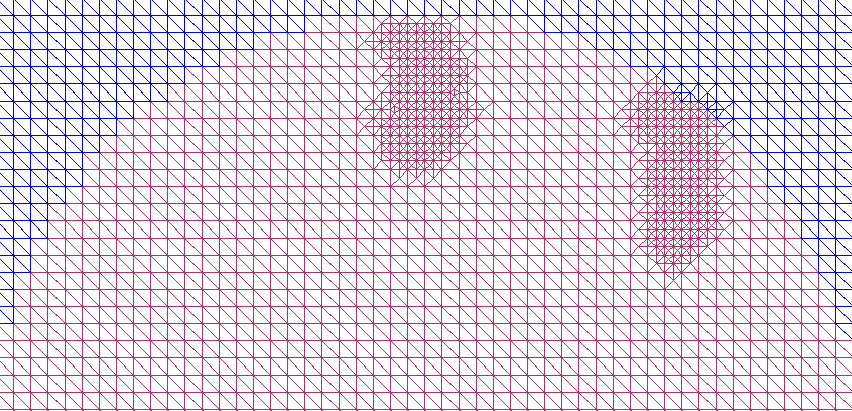}} \\
three times refined  mesh &  
     \end{tabular}} \\
  \hline
  $\rho(x), \delta=3\%$  & $ \rho(x), \delta=10\%$  &  $p(x), \delta= 3\%$ &  $p(x), \delta= 10\%$ \\
\hline
 {\includegraphics[scale=0.2, clip=true,]{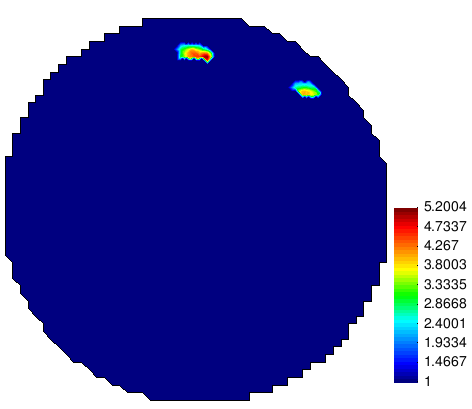}} &
 {\includegraphics[scale=0.2, clip=true,]{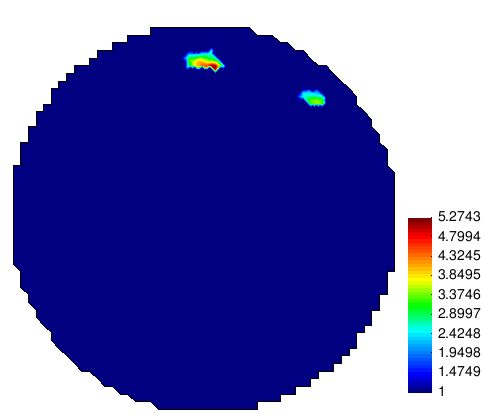}}     &
 {\includegraphics[scale=0.2,  clip=true,]{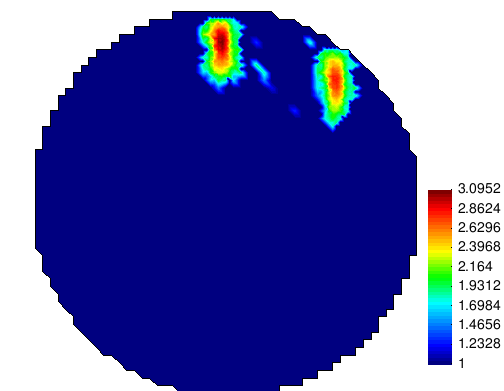}}     &
 {\includegraphics[scale=0.2,  clip=true,]{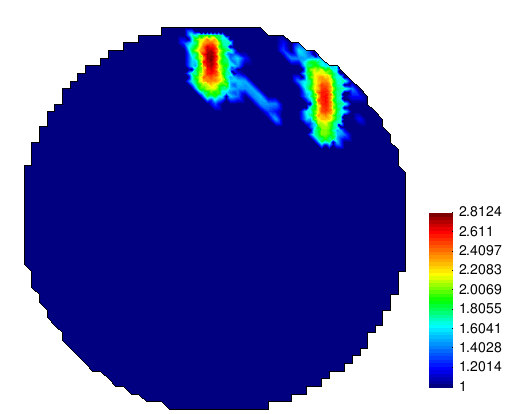}}
   \\
   \hline
\end{tabular}
 \end{center}
 \caption{ \protect\small \emph{Test 1: reconstructions obtained on
     three times adaptively refined mesh for different noise levels
     $\delta$ in data. }}
 \label{fig:rec_test1adapt}
 \end{figure}

 \begin{figure}
 \begin{center}
   \begin{tabular}{|c|c|c|c|}
     \hline
     \multicolumn{4}{|c|}
     {\begin{tabular}{cc}
         adaptively  refined meshes& zoomed \\
         \hline
         \\
         {\includegraphics[scale=0.27, clip=true,]{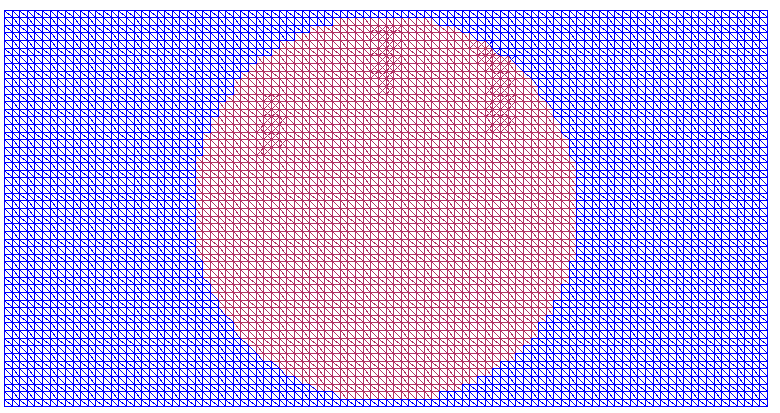}} &
         {\includegraphics[scale=0.23, trim = 0.0cm 0.0cm 4.0cm 0.0cm, clip=true,]{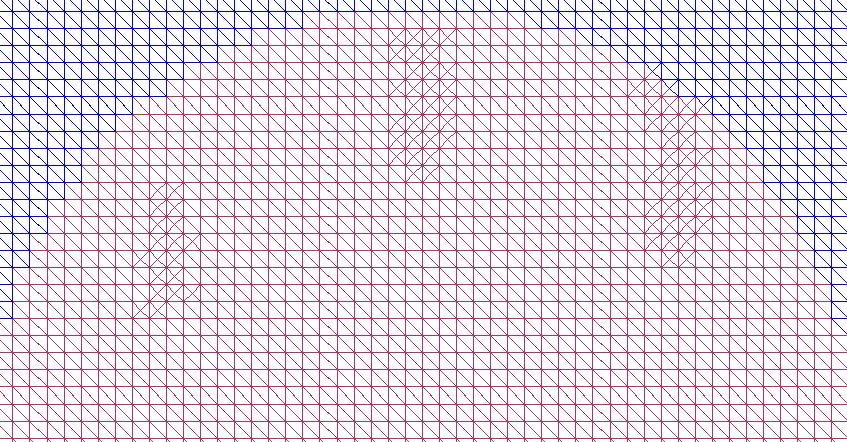}}    \\
once refined mesh& \\
   {\includegraphics[scale=0.27, clip=true,]{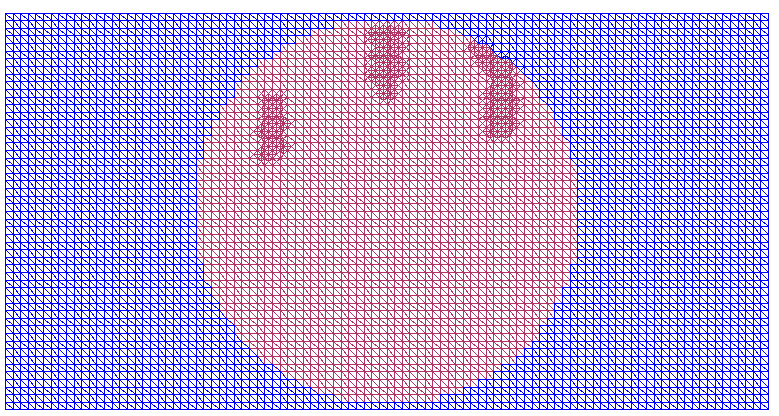}} &
   {\includegraphics[scale=0.23, trim = 0.0cm 0.0cm 4.0cm 0.0cm, clip=true,]{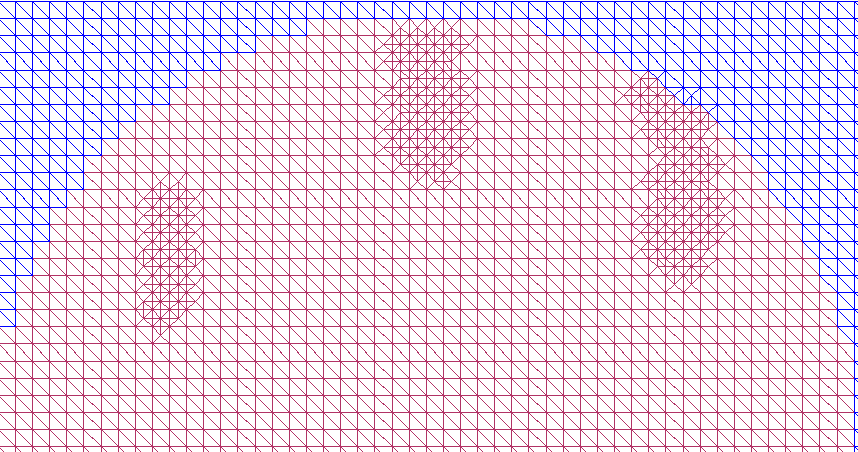}} \\
twice refined mesh&  \\
   {\includegraphics[scale=0.27, clip=true,]{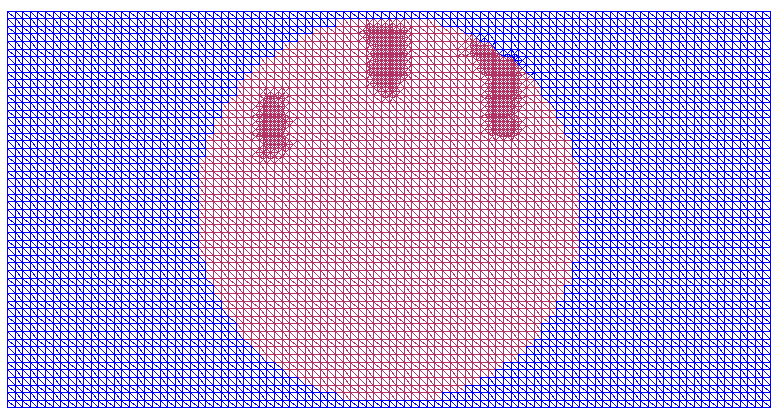}} &
   {\includegraphics[scale=0.23, trim = 0.0cm 0.0cm 4.0cm 0.0cm, clip=true,]{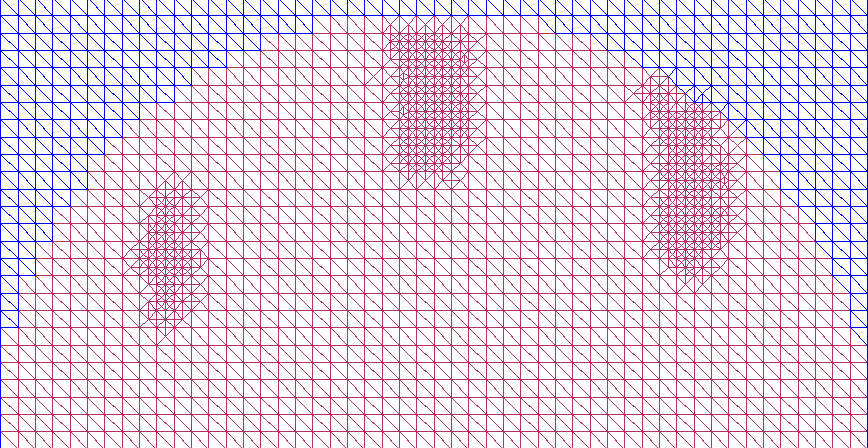}} \\
three times refined mesh & 
     \end{tabular}}
     \\
  \hline
  $\rho(x), \delta=3\%$  & $ \rho(x), \delta=10\%$  &  $p(x), \delta= 3\%$ &  $p(x), \delta= 10\%$ \\
\hline
 {\includegraphics[scale=0.2, clip=true,]{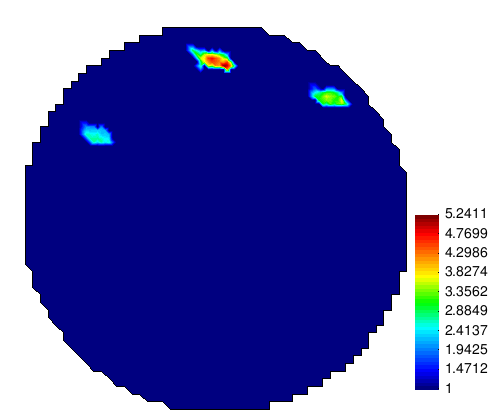}} &
 {\includegraphics[scale=0.2, clip=true,]{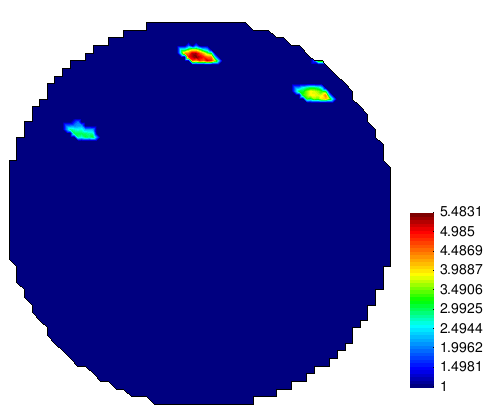}}     &
 {\includegraphics[scale=0.2,  clip=true,]{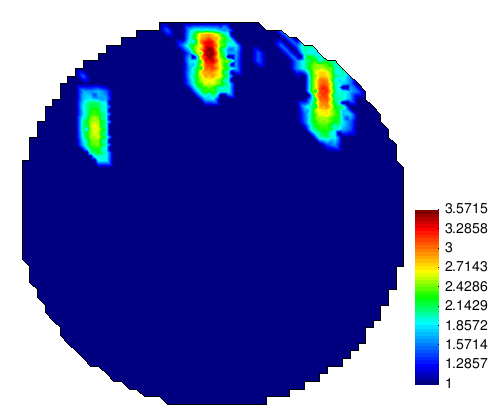}}     &
 {\includegraphics[scale=0.2,  clip=true,]{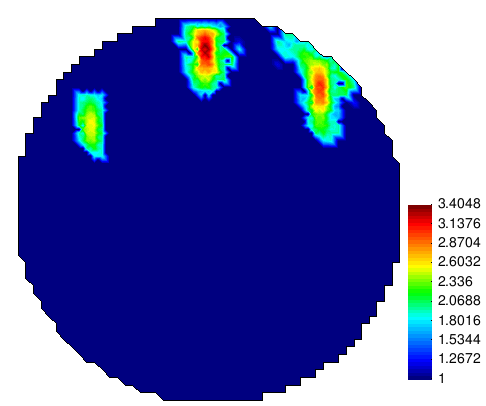}}
   \\
   \hline
\end{tabular}
 \end{center}
 \caption{ \protect\small \emph{Test 2: reconstructions obtained on
     two times adaptively refined mesh for different noise levels
     $\delta$ in data. }}
 \label{fig:rec_test2adapt}
 \end{figure}

 \begin{figure}
 \begin{center}
   \begin{tabular}{|c|c|c|c|}
     \hline
     \multicolumn{4}{|c|}
     {\begin{tabular}{cc}
         adaptively  refined meshes& zoomed \\
         \hline
         \\
         {\includegraphics[scale=0.27, clip=true,]{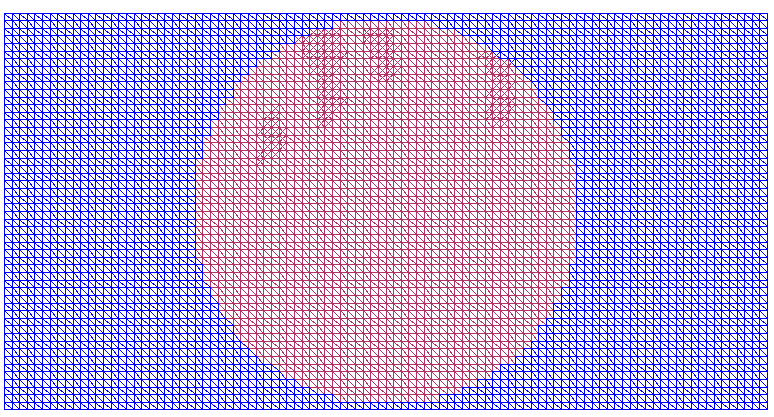}} &
         {\includegraphics[scale=0.23, trim = 0.0cm 0.0cm 4.0cm 0.0cm, clip=true,]{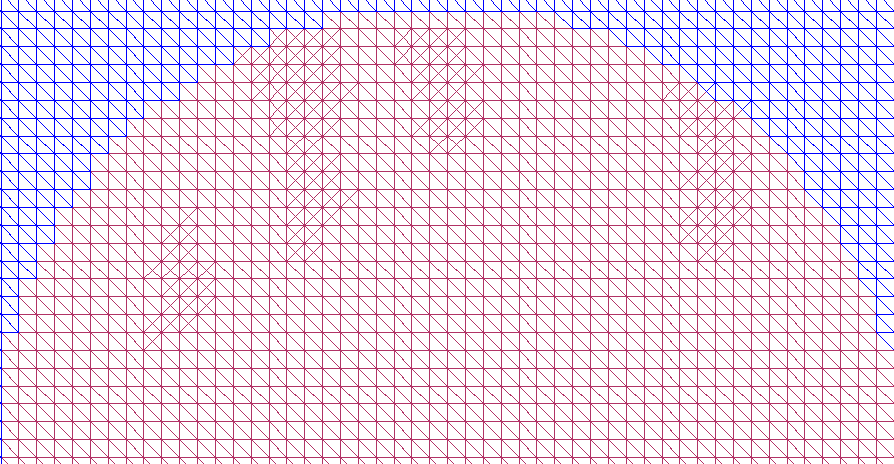}}    \\
once refined mesh & \\
   {\includegraphics[scale=0.27, clip=true,]{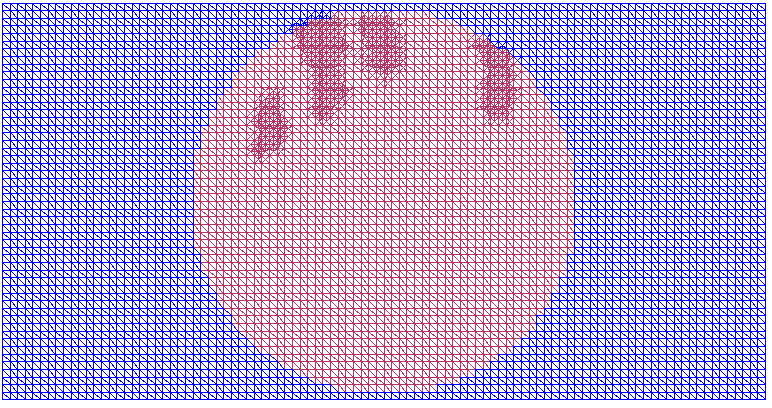}} &
   {\includegraphics[scale=0.23, trim = 0.0cm 0.0cm 4.0cm 0.0cm, clip=true,]{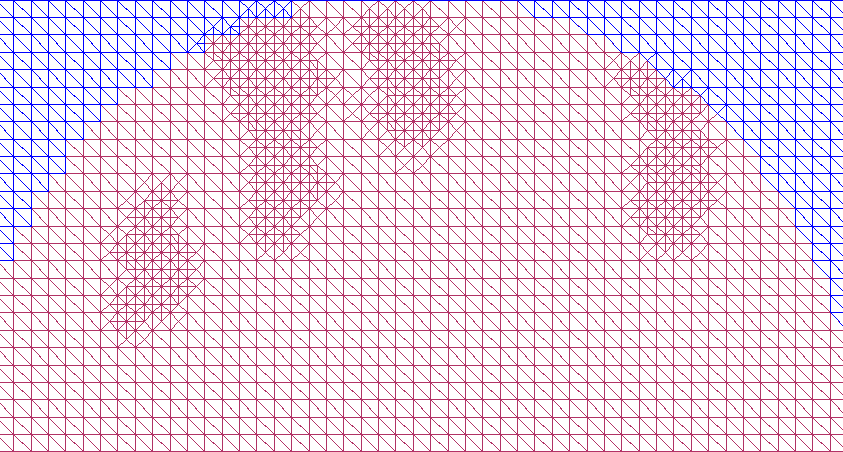}} \\
twice refined mesh & \\
   {\includegraphics[scale=0.27, clip=true,]{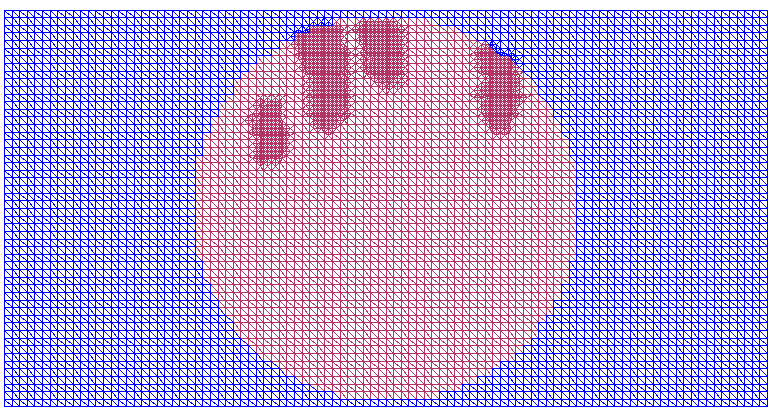}} &
   {\includegraphics[scale=0.23, trim = 0.0cm 0.0cm 4.0cm 0.0cm, clip=true,]{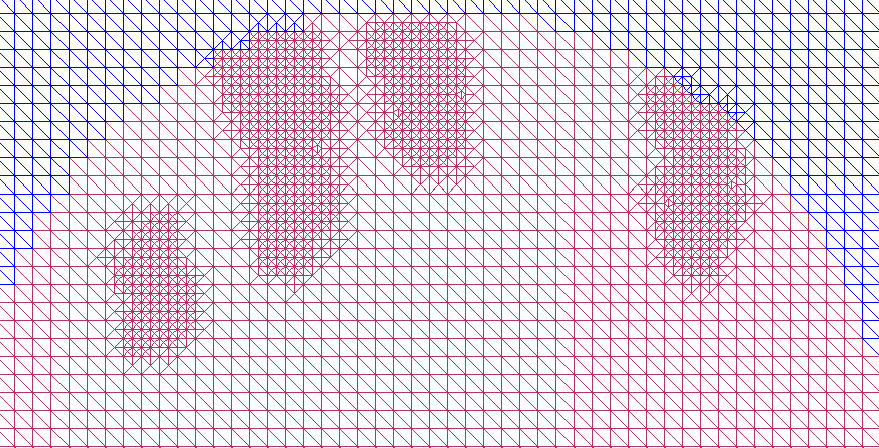}} \\
three times refined mesh & 
     \end{tabular}}
     \\
  \hline
  $\rho(x), \delta=3\%$  & $ \rho(x), \delta=10\%$  &  $p(x), \delta= 3\%$ &  $p(x), \delta= 10\%$ \\
\hline
 {\includegraphics[scale=0.2, clip=true,]{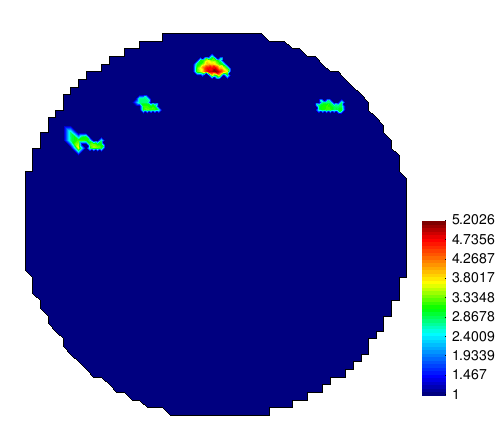}} &
 {\includegraphics[scale=0.2, clip=true,]{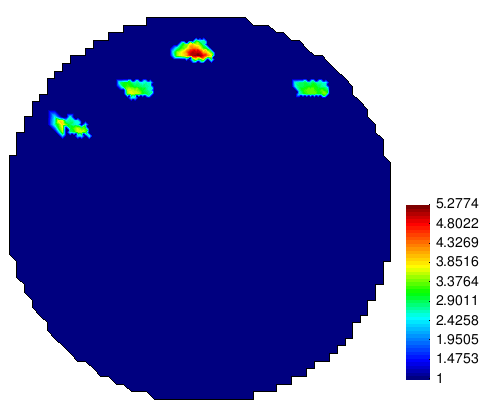}}     &
 {\includegraphics[scale=0.2,  clip=true,]{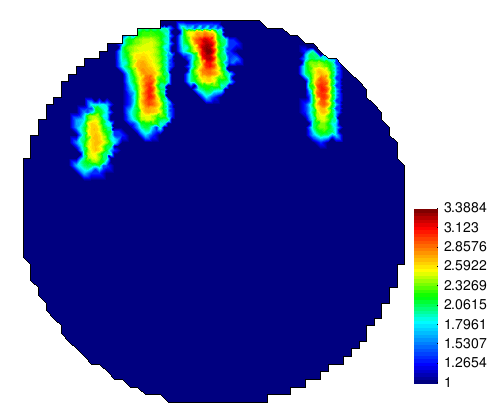}}     &
 {\includegraphics[scale=0.2,  clip=true,]{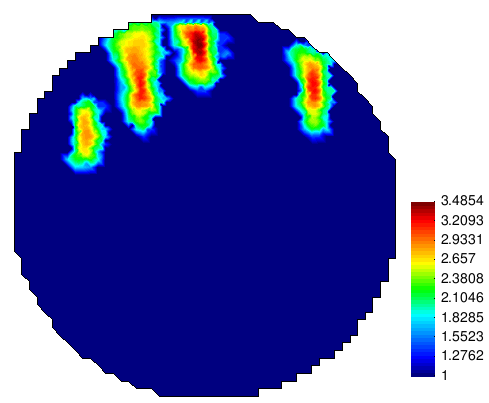}}
   \\
   \hline
\end{tabular}
 \end{center}
 \caption{ \protect\small \emph{Test 3: reconstructions obtained on
     three times adaptively refined mesh for different noise levels
     $\delta$ in data. }}
 \label{fig:rec_test3adapt}
 \end{figure}

We set the dimensionless computational domain
$\Omega$ in the domain decomposition  as
 \begin{equation*}
 \Omega = \left\{ x= (x_1,x_2): x_1 \in (-1.1, 1.1), x_2  \in (-0.62,0.62)\right\},
 \end{equation*}
and  the  domain $\Omega_{FEM}$ as
 \begin{equation*}
 \Omega_{FEM} = \left\{ x= (x_1,x_2): x_1 \in (-1.0,1.0), x_2 \in (-0.52,0.52) \right\}.
 \end{equation*}
 We choose the mesh size $h=0.02$ in $\Omega = \Omega_{FEM} \cup
 \Omega_{FDM}$, as well as in the overlapping regions between FE/FD
 domains.

 We assume that our two  functions $\rho(x), p(x)$ belongs to the set
 of admissible parameters
 \begin{equation}\label{admpar}
 \begin{split}
   M_{\rho} &= \{ \rho\in C^2(\overline{\Omega });~ 1\leq \rho(x)\leq 10\}, \\
   M_{p} &= \{ p \in C^2(\overline{\Omega }); ~1\leq p(x)\leq 5\}.
 \end{split}
 \end{equation}

 We define now our  coefficient inverse problem which we use in computations.

 \textbf{Inverse Problem (IP) } \emph{Assume that the functions}
 $\rho(x), p(x)$ \emph{\ of the model problem \eqref{model1} are
   unknown. Let these functions satisfy
   conditions (\ref{admpar},) and $\rho(x)=1, p(x)=1$ in the domain
 }$\Omega \backslash \Omega_{\rm FEM}$\emph{. Determine the functions}
 $\rho(x), p(x) $\emph{\ for }$x\in \Omega \backslash \Omega_{\rm FDM},$
 \emph{\ assuming that the following function }$\tilde
 u\left(x,t\right) $\emph{\ is known}
\begin{equation}\label{backscat}
  u \left(x,t\right) = \tilde u \left(x,t\right) , \forall \left( x,t\right)
  \in  \Gamma_1.
\end{equation}

To determine both coefficients $\rho(x), p(x)$  in inverse problem \textbf{{IP} }
we  minimize the
following Tikhonov functional
\begin{equation}
\begin{split}
J( \rho(x), p(x)) &:= J(u, \rho, p) =
\frac{1}{2} \int_{\Gamma_1} (u - \tilde{u} )^2 z_{\delta }(t) ds dt
 \\
&+\frac{1}{2} \alpha_1  \int_{\Omega}( \rho -  \rho_0)^2~~ dx
+ \frac{1}{2} \alpha_2  \int_{\Omega}( p -   p_0)^2~~ dx.
\label{functional}
\end{split}
\end{equation}
Here, $\tilde{u}$ is the observed function $u$ in time at the
backscattered boundary $\partial_1 \Omega$, the function $u$ satisfy
the equations (\ref{model1}) and thus depends on $\rho, p$, $\rho_0,
p_0$ are the initial guesses for $\rho, p$, correspondingly, and
$\alpha_i, i=1,2$, are regularization parameters.  We take $\rho_0=1,
p_0=1$ at all points of the computational domain since previous
computational works \cite{BAbsorb,BNAbsorb,BMaxwell,BH} as well as
experimental works of \cite{litman1,meaney1} have shown that a such
choice gives good results of reconstruction.  Here, $z_{\delta }(t)$
is a cut-off function
chosen as in \cite{BAbsorb, BNAbsorb,BH}.
This function is introduced to ensure the compatibility
conditions at $\overline{\Omega}_{T}\cap \left\{ t=T\right\} $ for the
adjoint problem, see details in \cite{BAbsorb, BNAbsorb,BH}.

To solve the minimization problem we take into
account conditions \eqref{admpar} and introduce the
Lagrangian
\begin{equation}\label{lagrangian}
\begin{split}
L(v) &= J(u, \rho, p)
+ \int_{\Omega} \int_0^T \lambda \Big (\rho \frac{\partial^2 u}{\partial t^2} - \ddd ( p \nabla  u)  \Big)~ dxdt,
\end{split}
\end{equation}
where $v=(u,\lambda, \rho, p)$.  Our goal is
 to find a stationary point of the Lagrangian with respect to $v$
satisfying $\forall \bar{v}= ( \bar{u}, \bar{\lambda},
\bar{\rho}, \bar{p})$
\begin{equation}
 L'(v; \bar{v}) = 0 ,  \label{scalar_lagr2}
\end{equation}
where $ L^\prime (v;\cdot )$ is the Jacobian of $L$ at $v$.  To find
optimal parameters $\rho, p$ from \eqref{scalar_lagr2} we use the
conjugate gradient method with iterative choice of the regularization
parameters $\alpha_j, j=1,2$, in \eqref{functional}.  More precisely, in all
our computations we choose the regularization parameters iteratively
 as  was proposed  in \cite{BKS},
such that $\alpha^n_j=\alpha^0_j(n+1)^{-q}$, where $n$ is the number
of iteration in the conjugate gradient method, $q \in (0,1)$ and
$\alpha^0_j$ are initial guesses for $\alpha_j, j=1,2$. Similarly with
\cite{Klibanov_Bakushinsky_Beilina} we take
$\alpha_j=\delta^{\zeta}$, where $\delta$ is the noise level and
$\zeta$ is a small number taken in the interval $(0,1)$.
Different techniques for the
  computation of a regularization parameter are presented in works
  \cite{Engl, IJT11, KNS, tikhonov}, and checking of performance of these
  techniques for the solution of our inverse problem can be challenge
  for our future research.

 To generate backscattered data we solve the model problem
 (\ref{model1}) in time $T=[0,2.0]$ with the time step $\tau=0.002$
 which satisfies to the CFL condition \cite{CFL67}.  In order to check
 performance of the reconstruction algorithm we supply simulated
 backscattered data at $\partial_1 \Omega$ by additive, as in
 \cite{BAbsorb,BNAbsorb,BH},
 noise
 $\delta=3\%, 10\%$.
 Similar results of reconstruction are obtained
 for random noise and they
will be presented in the forthcoming publication.

\begin{table}[tbp]
{\footnotesize Table 1. \emph{Computational results of the
    reconstructions on a coarse and on adaptively refined meshes together with computational
    errors in the maximal contrast of $\rho(x), p(x)$ in percents. Here,
    $N_{\rho}^j, N_{p}^j$ denote the final number of
    iterations in the conjugate gradient method on $j$ times refined mesh for reconstructed
    functions $\rho$ and $p$, respectively.}}  \par
\vspace{2mm}
\centerline{
  \begin{tabular}{|c|c|}
     \hline
     \multicolumn{2}{|c|}
                 {Coarse mesh}
                 \\
 \hline
   $\delta=3\%$ &  $\delta = 10\%$
 \\
\begin{tabular}{l|l|l|l} \hline
Case & $\max_{\Omega_1} \rho$ &  error, \% & $N_{\rho}^0$  \\ \hline
Test 1 & 4.13  & 17.4 & 13   \\
Test 2 & 4.38 &  12.4 & $15$    \\
Test 3 & 5.14  &  2.8 &  $16$   \\
Test 4 &  4.12 &  17.6 &  $14$  \\
\end{tabular}
 &
\begin{tabular}{l|l|l|l} \hline
Case & $\max_{\Omega_1} \rho$ &    error, \% & $N_{\rho}^0$  \\ \hline
Test 1 & 3.74  & 25.2   &  12    \\
Test 2 & 3.84  & 23.2   &  $13$  \\
Test 3 & 5.08  & 1.6   &   $16$  \\
Test 4 & 3.9   &  22  &   $13$ \\
\end{tabular}
\\
   \hline
\begin{tabular}{l|l|l|l} \hline
Case & $\max_{\Omega_1} p$ &  error, \% & $N_p^0$  \\ \hline
Test 1 & 3.09 & 3 &  $13$    \\
Test 2 & 3.63  &  21   &    $15$    \\
Test 3 & 3.63 &  21 &  $16$    \\
Test 4 & 3.4 &  13.3 &   $14$   \\
\end{tabular}
 &
\begin{tabular}{l|l|l|l} \hline
Case & $\max_{\Omega_1} p$ &    error, \% & $N_p^0$  \\ \hline
Test 1 & 2.9  &  3.33  &   $12$   \\
Test 2 & 3.16 & 5.33   &  $13$    \\
Test 3 & 3.74 &  24.67  &  $16$    \\
Test 4 & 3.24  &   8  &   13 \\
\end{tabular}
\\
\hline
\multicolumn{2}{|c|}
                 {Adaptively refined mesh}
                 \\
\begin{tabular}{l|l|l|l} \hline
Case & $\max_{\Omega_1} \rho$ &  error, \% & $N_{\rho}^j$  \\ \hline
Test 1 & 5.2  & 4 &  $N_{\rho}^3 =9$   \\
Test 2 & 5.24  & 4.8  & $N_{\rho}^2 = 6 $     \\
Test 3 & 5.2 & 4 & $N_{\rho}^3= 1$   \\
Test 4 &  5.5  &  10 & $N_{\rho}^3= 8$   \\
\end{tabular}
 &
\begin{tabular}{l|l|l|l} \hline
Case & $\max_{\Omega_1} \rho$ &    error, \% & $N_{\rho}^j$  \\ \hline
Test 1 & 5.3 & 6  &  $N_{\rho}^3 = 7$   \\
Test 2 & 5.5 & 10   &  $N_{\rho}^2 = 10$   \\
Test 3 & 5.28 & 5.6 &  $N_{\rho}^3 =1$   \\
Test 4 & 5.36  & 7.2  &   $N_{\rho}^3 = 8$  \\
\end{tabular}
\\
\begin{tabular}{l|l|l|l} \hline
Case & $\max_{\Omega_1} p$ &  error, \% & $N_p^j$  \\ \hline
Test 1 & 3.1 & 3.33 & $N_{p}^3= 9$   \\
Test 2 & 3.57  & 19  & $N_{p}^2 = 6$    \\
Test 3 &  3.39 & 13 & $N_{p}^3=  1$   \\
Test 4 &  3.4 & 13.3 & $N_{p}^3=14  $  \\
\end{tabular}
 &
\begin{tabular}{l|l|l|l} \hline
Case & $\max_{\Omega_1} p$ &    error, \% & $N_p^j$  \\ \hline
Test 1 &  2.8 &  6.67 & $N_{p}^3=7$  \\
Test 2 &  3.4  & 13.3  & $N_{p}^2=9$   \\
Test 3 &  3.49  & 16.3 & $N_{p}^3=1$  \\
Test 4 &   3.26 &  8.67 & $N_{p}^3=10$  \\
\end{tabular}
\\
\hline
\end{tabular}}
\end{table}

 \subsection{Test 1}

\label{sec:test1}

In this test we present numerical results of the simultaneous
reconstruction of two functions $\rho(x)$ and $p(x)$ given by
\begin{equation}\label{2gaussians}
\begin{split}
\rho(x) = 1.0 &+ 4.0 \cdot {\rm e}^{-((x_1 -0.3)^2 + {(x_2 - 0.3)}^2)/0.001} \\
&+ 4.0 \cdot  {\rm e}^{-(x_1^2 + {(x_2 - 0.4)}^2)/0.001} , \\
p(x) = 1.0 &+ 2.0 \cdot {\rm e}^{-((x_1 -0.3)^2 + {(x_2 - 0.3)}^2)/0.001} \\
&+ 2.0 \cdot  {\rm e}^{-(x_1^2 + {(x_2 - 0.4)}^2)/0.001}, \\
\end{split}
\end{equation}
 which are presented in Figure ~\ref{fig:exact_gaussians}.

 Figures \ref{fig:rec_test} show results of the reconstruction on a
 coarse mesh with additive noise $\delta= 3\%, 10\%$ in data.
 We observe that the location of both functions $\rho, p$ given by
 (\ref{2gaussians}) is imaged correctly. We refer to Table 1 for the
 reconstruction of the contrast in both  functions.

 To improve contrast and shape of the reconstructed functions
 $\rho(x)$ and $p(x)$ we run computations again using an adaptive
 conjugate gradient method similar to the one of \cite{BH}. Figure
 \ref{fig:rec_test1adapt}  and Table 1 show results of reconstruction on the three
 times locally refined mesh. We observe that we achieve better contrast
 for both functions  $\rho(x)$ and $p(x)$, as well as better shape
  for the function $\rho(x)$.

 \subsection{Test 2}

\label{sec:test2}

In this test we present numerical results of the reconstruction of
the functions $\rho(x)$ and $p(x)$ given by three Gaussians shown in
Figure ~\ref{fig:exact_gaussians} and given by
\begin{equation}\label{3gaussians}
\begin{split}
\rho(x) = 1.0 &+ 4.0 \cdot {\rm e}^{-((x_1 -0.3)^2 + {(x_2 - 0.3)}^2)/0.001} \\
&+ 4.0 \cdot  {\rm e}^{-(x_1^2 + {(x_2 - 0.4)}^2)/0.001}\\
&+ 4.0 \cdot  {\rm e}^{-((x_1 + 0.3)^2 + {(x_2 - 0.2)}^2)/0.001}, \\
p(x) = 1.0 &+ 2.0 \cdot {\rm e}^{-((x_1 -0.3)^2 + {(x_2 - 0.3)}^2)/0.001} \\
&+ 2.0 \cdot  {\rm e}^{-(x_1^2 + {(x_2 - 0.4)}^2)/0.001} \\
&+ 2.0 \cdot  {\rm e}^{-((x_1 + 0.3)^2 + {(x_2 - 0.2)}^2)/0.001}.
\end{split}
\end{equation}

 Figures \ref{fig:rec_test} show results of the reconstruction on a
 coarse mesh with additive noise $\delta= 3\%, 10\%$ in data.
 We observe that the location of three Gaussians for both functions
 $\rho, p$ is imaged correctly, see Table 1 for the reconstruction of
 contrast in these functions.

 To improve contrast and shape of the reconstructed functions
 $\rho(x)$ and $p(x)$ we run computations again using an adaptive
 conjugate gradient method similar to the one  of \cite{BH}. Figure
 \ref{fig:rec_test2adapt}  and Table 1 show results of reconstruction on the two
 times locally refined mesh.
 We observe that we achieve better contrast
 for both functions  $\rho(x)$ and $p(x)$, as well as better shape
  for the function $\rho(x)$. Results on the three times refined mesh
 were similar to the results obtained on a two times refined mesh, and
 we  are not presenting them here.

 \subsection{Test 3}

\label{sec:test3}

This test presents numerical results of the reconstruction of
the functions $\rho(x)$ and $p(x)$ given by four different  Gaussians shown in
Figure ~\ref{fig:exact_gaussians} and given by
\begin{equation}\label{4difgaussians}
\begin{split}
\rho(x) = 1.0 &+ 4.0 \cdot {\rm e}^{-((x_1 -0.3)^2 + {(x_2 - 0.3)}^2)/0.001} \\
&+ 4.0 \cdot  {\rm e}^{-(x_1^2 + {(x_2 - 0.4)}^2)/0.001}\\
&+ 4.0 \cdot  {\rm e}^{-((x_1 + 0.3)^2 + {(x_2 - 0.2)}^2)/0.001} \\
&+ 4.0 \cdot  {\rm e}^{-((x_1 + 0.15)^2 + {(x_2 - 0.3)}^2)/0.001}, \\
p(x) = 1.0 &+ 2.0 \cdot {\rm e}^{-((x_1 -0.3)^2 + {(x_2 - 0.3)}^2)/0.001} \\
&+ 2.0 \cdot  {\rm e}^{-(x_1^2 + {(x_2 - 0.4)}^2)/0.001} \\
&+ 2.0 \cdot  {\rm e}^{-((x_1 + 0.3)^2 + {(x_2 - 0.2)}^2)/0.001} \\
&+ 2.0 \cdot  {\rm e}^{-((x_1 + 0.15)^2 + {(x_2 - 0.3)}^2)/0.001}.
\end{split}
\end{equation}

 Figures \ref{fig:rec_test} show results of the reconstruction of four
 Gaussians on a coarse mesh with additive noise $\delta= 3\%, 10\%$ in
 data.  We have obtained similar results as in the two previous tests:
 the location of four Gaussians for both functions $\rho, p$ already
 on a coarse mesh is imaged correctly.  However, as follows from the
 Table 1, the contrast should be improved.  Again, to improve the
 contrast and shape of the Gaussians we run an adaptive conjugate
 gradient method similar to one of  \cite{BH}. Figure \ref{fig:rec_test3adapt}
 shows results of reconstruction on the three times locally refined
 mesh. Using Table 1 we observe that we achieve better contrast for
 both functions $\rho(x)$ and $p(x)$, as well as better shape for the
 function $\rho(x)$.

 \subsection{Test 4}

\label{sec:test4}

In this test we tried to reconstruct four  Gaussians shown in
Figure ~\ref{fig:exact_gaussians} and given by
\begin{equation}\label{4gaussians}
\begin{split}
\rho(x) = 1.0 &+ 4.0 \cdot {\rm e}^{-((x_1 -0.3)^2 + {(x_2 - 0.3)}^2)/0.001} \\
&+ 4.0 \cdot  {\rm e}^{-(x_1^2 + {(x_2 - 0.4)}^2)/0.001}\\
&+ 4.0 \cdot  {\rm e}^{-((x_1 + 0.3)^2 + {(x_2 - 0.2)}^2)/0.001}  \\
&+ 4.0 \cdot  {\rm e}^{-(x_1^2 + {(x_2 - 0.2)}^2)/0.001}, \\
p(x) = 1.0 &+ 2.0 \cdot {\rm e}^{-((x_1 -0.3)^2 + {(x_2 - 0.3)}^2)/0.001} \\
&+ 2.0 \cdot  {\rm e}^{-(x_1^2 + {(x_2 - 0.4)}^2)/0.001} \\
&+ 2.0 \cdot  {\rm e}^{-((x_1 + 0.3)^2 + {(x_2 - 0.2)}^2)/0.001} \\
&+ 2.0 \cdot  {\rm e}^{-(x_1^2 + {(x_2 - 0.2)}^2)/0.001}.
\end{split}
\end{equation}
We observe that two Gaussians in this example are located one under another one.
Thus, backscattered data from these two Gaussians will be superimposed and
thus, we expect to reconstruct only three Gaussians from four.

 Figure \ref{fig:rec_test} shows results of the reconstruction of
 these four Gaussians on a coarse mesh with additive noise $\delta=
 3\%, 10\%$ in data.  As expected, we could reconstruct only three
 Gaussians from four, see Table 1 for reconstruction of the contrast
 in them.  Even application of the adaptive algorithm can not give us
 the fourth Gaussian. However, the contrast in the reconstructed
 functions is improved, as in Test 3.

\section{Conclusions}\label{S5}

\label{sec:concl}

In this work we present theoretical and numerical studies of the
reconstruction of two space-dependent functions $\rho(x)$ and $p(x)$
in a hyperbolic problem.\\ In the theoretical part of this work we
derive a local Carleman estimate which allows to obtain a conditional
Lipschitz stability inequality for the inverse problem formulated in
section \ref{S1}. This stability is very important for our subsequent
numerical reconstruction of the two unknown functions $\rho(x)$ and
$p(x)$ in the hyperbolic model \eqref{model1}.\\ In the numerical part
we present a computational study of the simultaneous reconstruction of
two functions $\rho(x)$ and $p(x)$ in a hyperbolic problem
(\ref{model1}) from backscattered data using an adaptive domain
decomposition finite element/difference method similar to one
developed in \cite{BAbsorb, BH}.  In our numerical tests, we have
obtained stable reconstruction of the location and contrasts of both
functions $\rho(x)$ and $p(x)$ for noise levels $\delta = 3 \%, 10 \%$
in backscattered data.  Using results of Table 1 and Figures
\ref{fig:rec_test1adapt}--\ref{fig:rec_test3adapt} we can conclude,
that an adaptive domain decomposition finite element/finite difference
algorithm significantly improves qualitative and quantitative results
of the reconstruction obtained on a coarse mesh.

\section*{Acknowledgments}

The research of L. B.  is partially supported by the sabbatical
programme at the Faculty of Science, University of Gothenburg, Sweden.
The research of M.C. is partially supported by the guest programme of
the Department of Mathematical Sciences at Chalmers University of
Technology and Gothenburg University, Sweden. Research of M.Y. is
partially supported by Grant-in-Aid for Scientific Research (S)
15H05740 of Japan Society for the Promotion of Science.


\begin{thebibliography}{99}

\bibitem{BKS} A. Bakushinsky, M. Y. Kokurin, and A. Smirnova,
\emph{Iterative Methods for Ill-posed Problems},
De Gruyter, Berlin, 2011.

\bibitem{BMaxwell} L.Beilina,
\newblock Adaptive Finite Element Method for a coefficient inverse problem for
the Maxwell's system,  \emph{Applicable Analysis}, {\bf 90} (2011), 1461--1479.

\bibitem{BAbsorb} L.~Beilina,
Domain Decomposition finite
element/finite difference method for the conductivity
reconstruction in a hyperbolic equation, \emph{Communications in
Nonlinear Science and Numerical Simulation}, Elsevier, 2016,
doi:10.1016/j.cnsns.2016.01.016



 \bibitem{B}   L.~Beilina,
\newblock Adaptive hybrid FEM/FDM methods for inverse  scattering
problems,  \emph{ Inverse Problems and Information  Technologies}, {\bf 1}(3), 73--116, 2002.

\bibitem{BCL} L. Beilina, M. Cristofol, and S. Li, Uniqueness and
stability of time and space-dependent conductivity in a hyperbolic
cylindrical domain, \emph{arXiv:1607.01615}.

\bibitem{BCN}  L. Beilina, M. Cristofol, and K. Niinim\"aki, Optimization
approach for the simultaneous reconstruction of the dielectric
permittivity and magnetic permeability functions from limited
observations, \emph{Inverse Problems and Imaging}, {\bf 9} (2015), 1-25.

\bibitem{BH}
\newblock L. Beilina and S. Hosseinzadegan,
\newblock  An adaptive
finite element method in reconstruction of coefficients in Maxwell's
equations from limited observations, \emph{ Applications of Mathematics}, {\bf 61(3)} (2016),  253--286.

\bibitem{BJ} L. Beilina and C. Johnson, A posteriori error estimation in
computational inverse scattering, \emph{Mathematical Models in Applied
Sciences}, {\bf 1} (2005), 23-35.

\bibitem{Bekli} L. Beilina and M.V. Klibanov, \emph{Approximate Global
Convergence and Adaptivity for Coefficient Inverse Problems}, Springer-Verlag,
Berlin, 2012.

\bibitem{BNAbsorb} L.~Beilina and K. Niinim\"aki, Numerical studies of
the Lagrangian approach for reconstruction of the conductivity in a
waveguide, \emph{ arXiv:1510.00499}, 2015.

\bibitem{Bel1} M. Bellassoued, Uniqueness and stability in determining
the speed of propagation of second order
hyperbolic equation with variable coefficients, \emph{Appl. Anal.} {\bf 83} (2004),
983-1014.

\bibitem{Bel2} M.Bellassoued, Global logarithmic stability in inverse
hyperbolic problem by arbitrary boundary observation, \emph{Inverse Problems}
{\bf 20} (2004), 1033-1052.

\bibitem{BCS12}
\newblock M. Bellassoued, M. Cristofol, and E. Soccorsi,
\newblock Inverse boundary value problem for the dynamical heterogeneous
Maxwell's system, \emph{Inverse Problems} {\bf 28} (2012), 095009.

\bibitem{BIY}
M. Bellassoued, O. Y. Imanuvilov, and M. Yamamoto, Inverse problem of
determining the density and two Lame coefficients by boundary data,
\emph{SIAM J. Math. Anal.} {\bf 40} (2008), 238-265.

\bibitem{BJY08}
M. Bellassoued, D. Jellali and M. Yamamoto,
Lipschitz stability in in an inverse problem for a hyperbolic equation
with a finite set of boundary data, \emph{Applicable Analysis} {\bf 87}
(2008), 1105-1119.

\bibitem{BelY1} M. Bellassoued and M. Yamamoto, Logarithmic stability in
determination of a coefficient in
an acoustic equation by arbitrary boundary observation, \emph{J. Math.
Pures Appl.}  {\bf 85} (2006), 193-224.

\bibitem{BeYa08} M. Bellassoued and M. Yamamoto,
Determination of a coefficient in the wave equation with
a single measurement, \emph{Appl. Anal.} {\bf 87} (2008), 901-920.

\bibitem{BY}
M. Bellassoued and M. Yamamoto,
\emph{Carleman Estimates and Applications to Inverse Problems for
Hyperbolic Systems}, Springer-Japan, to appear.

\bibitem{BK} A.L. Bugkheim and M.V.Klibanov, Global uniqueness of class of
multidimentional inverse problems, \emph{Soviet Math. Dokl.} {\bf 24} (1981),
244-247.

\bibitem{ChIsYZh}
J. Cheng, V. Isakov, M. Yamamoto, and Q. Zhou, Lipschitz stability
in the lateral Cauchy problem for elasticity system, \emph{J. Math. Kyoto
Univ. } {\bf 43} (2003), 475-501.

\bibitem{CFL67}
R. Courant, K. Friedrichs and H. Lewy,  On the partial differential equations
of mathematical
physics,  \emph{ Journal of Research and Development}, {\bf 11(2)}
(1967), 215--234.

\bibitem{Chow} Y. T. Chow and J. Zou, A numerical method for
  reconstructing the coefficient in a wave equation, \emph{Numerical
    Methods for Partial Differential Equations} {\bf 31} (2015),
  289--307.

\bibitem{Engl}  H. W. Engl,  M. Hanke and A. Neubauer, \emph{Regularization
of Inverse Problems}, Kluwer, Boston, 2000.

\bibitem{Ima2} O. Y. Imanuvilov, On Carleman estimates for hyperbolic
equations, \emph{Asymptotic Analysis} {\bf 32} (2002), 185-220.

\bibitem{ImIsY} O. Y. Imanuvilov, V. Isakov and M. Yamamoto,
An inverse problem for the dynamical Lam\'e system with two sets of boundary
data, \emph{Comm. Pure Appl. Math.} {\bf 56} (2003), 1366-1382.

\bibitem{IY2} O. Y. Imanuvilov and M. Yamamoto, Global uniqueness and
stability in determining coefficients of wave equations,  \emph{Comm. Partial
Differential Equations} {\bf 26} (2001), 1409-1425.

\bibitem{IY3} O. Y. Imanuvilov and M. Yamamoto, Global Lipschitz stability in
  an inverse hyperbolic problem by interior observations,
  \emph{Inverse Problems}
{\bf 17} (2001), 717-728.

\bibitem{IY4} O. Y. Imanuvilov and M. Yamamoto, Determination of a
  coefficient in an acoustic equation with single measurement,
  \emph{Inverse Problems} {\bf 19} (2003), 157-171.

\bibitem{Is1} V. Isakov, \emph{Inverse Problems for Partial Differential Equations},
Springer-Verlag, Berlin, 1998, 2006.

\bibitem{IJT11}  K.~Ito, B.~Jin, and T.~Takeuchi, Multi-parameter
Tikhonov regularization, \emph{Methods and Applications of Analysis},
{\bf 18} (2011), 31-46.

\bibitem{KNS}  B. Kaltenbacher, A. Neubauer, and O. Scherzer.
\emph{Iterative Regularization Methods for Nonlinear Problems},
de Gruyter, Berlin, 2008.

\bibitem{Kli1}
M. V. Klibanov, Inverse problems in the "large" and Carleman bounds,
\emph{Differential Equations}, {\bf 20} (1984), 755-760.

\bibitem{Kli2} M. V. Klibanov, Inverse problems and Carleman estimates,
\emph{Inverse Problems}, {\bf 8}(1992), 575-596.

\bibitem{Kli3} M. V. Klibanov,  Carleman estimates for global uniqueness,
stability and numerical methods for coefficient inverse problems,
\emph{J. Inverse Ill-Posed Probl.}, {\bf 21} (2013), 477-560.

\bibitem{Klibanov_Bakushinsky_Beilina}
M. V. Klibanov, A. B. Bakushinsky, L. Beilina,
Why a minimizer of the Tikhonov functional is closer to the exact solution
than the first guess, \emph{Journal of Inverse and Ill - Posed Problems},
{\bf 19} (2011), 83-105.

\bibitem{KT} M. V. Klibanov and A. Timonov, \emph{Carleman Estimates for
Coefficient  Inverse Problems and Numerical Applications}, VSP, Utrecht, 2004.

\bibitem{KY} M. V. Klibanov and M. Yamamoto, Lipschitz stability of an inverse
problem for an acoustic equation, \emph{Appl. Anal.}, {\bf 85} (2006), 515-538.

\bibitem{L05}
S. Li, An inverse problem for Maxwell's equations in bi-isotropic media,
\emph{SIAM Journal on Mathematical Analysis}, {\bf 37} (2005), 1027--1043.

\bibitem{LY05}
S. Li and M. Yamamoto,
An inverse source problem for Maxwell's equations in anisotropic media,
\emph{Applicable Analysis}, {\bf 84} (2005), 1051--1067.

\bibitem{LY07}
S. Li and M. Yamamoto,
An inverse problem for Maxwell's equations in anisotropic media in two
dimensions, \emph{Chin. Ann. Math. Ser B} {\bf 28} (2007), 35-54.

\bibitem{L88} J.-L.  Lions, \emph{Controlabilit\'e Exacte, Perturbations et
Stabilisation des Syst\`eme Distribu\'es}, Masson, Paris, 1988.

\bibitem{LG72}  J.-L. Lions and E. Magenes,
\emph{Non-Homogeneous Boundary Value Problems and Applications}, Berlin, Springer,
1972.

\bibitem{litman1} C. Eyraud, J.-M. Geffrin, and A. Litman, 3-D imaging
  of a microwave absorber sample from microwave scattered field
  measurements, \emph{IEEE Microwave and Wireless Components Letters},
  {\bf 25(7)}(2015),  472--474.

\bibitem{meaney1}
T. M. Grzegorczyk, P. M. Meaney, P. A. Kaufman, R. M. diFlorio Alexander,
and K.D. Paulsen, Fast 3-d tomographic microwave imaging for breast
cancer detection, \emph{IEEE Trans Med Imaging}, {\bf 31} (2012), 1584--1592.

\bibitem{petsc}
\newblock PETSc, Portable, Extensible Toolkit for
Scientific Computation, http://www.mcs.anl.gov/petsc/

\bibitem{tikhonov} A. N. Tikhonov, A. V. Goncharsky, V. V. Stepanov, and
A. G.  Yagola,  \emph{Numerical Methods for the Solution of Ill-Posed Problems},
 Kluwer, London, 1995.

\bibitem{waves}  WavES, the software package, http://www.waves24.com~

\bibitem{Y99}
M. Yamamoto,  Uniqueness and stability in multidimensional hyperbolic
inverse problems, \emph{J. Math. Pures Appl.}, {\bf 78} (1999), 65-98.

\end{thebibliography}
\end{document}